\documentclass[11pt]{article}

\usepackage{listings}
\usepackage{pdflscape}
\usepackage{epsfig}
\usepackage{lscape}
\usepackage{longtable}
\usepackage{amsmath,amssymb,amsthm}
\usepackage{graphicx}
\usepackage{color}
\usepackage{placeins}
\usepackage{url}
\usepackage{cases}
\newcommand{\blue}[1]{\begin{color}{blue}#1\end{color}}

\newcommand{\red}[1]{\begin{color}{red}#1\end{color}}
\newcommand{\green}[1]{\begin{color}{green}#1\end{color}}
\DeclareMathOperator*{\argmin}{argmin}

\def\approxleq{ \kern3pt \mbox{\raisebox{.6ex}{$<$}} \kern-8pt
  \mbox{\raisebox{-.6ex}{$\sim$}} \kern5pt}
\def\mc{\multicolumn}
\def\norm#1{\|#1 \|}
\def\inprod#1#2{\langle#1,\,#2 \rangle}
\def\Inprod#1#2{\left\langle#1,\,#2 \right\rangle}

\def\cA{{\cal A}} \def\cB{{\cal B}}  \def\cK{{\cal K}}
 \def\cD{{\cal D}}  \def\cT{{\cal T}}
 \def\cU{{\cal U}}    
\def\cQ{{\cal Q}} \def\cM{{\cal M}} \def\cN{{\cal N}} \def\cW{{\cal W}}
\def\cS{{\cal S}}   \def\cF{{\cal F}}  
\def\cI{{\cal I}} \def\cX{{\cal X}}  \def\cZ{{\cal Z}}
 \def\cN{{\cal N}} \def\cL{{\cal L}}
   
\def\Sn{{\cal S}^n}

\def\nn{\nonumber}

\def\sig{\sigma}

\def\Range{\textup{Ran}}
\def\Diag{\textup{Diag}}

\def\dom{\textup{dom}}

\def\tZ{\widetilde{Z}}
\def\tS{\widetilde{S}} \def\tR{\widetilde{R}}
\def\tW{\widetilde{W}}
\def\ty{\tilde{y}}

\def\by{\bar{y}}

\newcommand{\T}{\mbox{\textrm{\tiny{T}}}}
\def\QSDPNAL {{\sc Qsdpnal }}

\newlength{\len}
\newtheorem{theorem}{Theorem}[section]
\newtheorem{prop}{Proposition}[section]
\newtheorem{lemma}{Lemma}[section]

\newtheorem{remark}{Remark}[section]
\newtheorem{assumption}{Assumption}

\begin{document}

\title{\bf QSDPNAL: A two-phase augmented Lagrangian method for convex quadratic semidefinite programming\footnote{An earlier version of this paper was  made available in arXiv as arXiv:1512.08872.}}

\author{Xudong Li\thanks{Department of Mathematics, National University of Singapore, 10 Lower Kent Ridge Road, Singapore ({\tt matlixu@nus.edu.sg}).
},
\; Defeng Sun\thanks{Department  of  Mathematics  and  Risk  Management  Institute, National University of Singapore, 10 Lower Kent Ridge Road, Singapore ({\tt matsundf@nus.edu.sg}).
This research is
supported in part by the Ministry of Education,
Singapore, Academic Research Fund under Grant R-146-000-194-112.
}
\  and \ Kim-Chuan Toh\thanks{Department of Mathematics, National University of Singapore, 10 Lower Kent Ridge Road, Singapore
({\tt mattohkc@nus.edu.sg}).  This research is
supported in part by the Ministry of Education,
Singapore, Academic Research Fund under Grant R-146-000-194-112.
}
}
\date{  December 17, 2016}
\maketitle

\begin{abstract}
 In this paper, we present a two-phase augmented Lagrangian method, called QSDPNAL, for solving convex quadratic semidefinite programming (QSDP) problems with constraints consisting of a large number of linear equality, inequality constraints, a simple convex polyhedral set constraint, and a positive semidefinite {cone} constraint.
A first order algorithm which relies on the inexact Schur complement based decomposition technique
is developed in QSDPNAL-Phase I with the aim of solving a QSDP problem to moderate accuracy or using it to generate a reasonably good initial point for the second phase. In QSDPNAL-Phase II, we design an augmented Lagrangian method (ALM) where the inner subproblem in each iteration is solved via inexact semismooth Newton based algorithms. Simple and implementable stopping criteria are 
{designed}
for the ALM.
Moreover, under mild conditions, we are able to 
{establish}
 the rate of convergence of the proposed algorithm and prove the R-(super)linear convergence 
of the KKT residual. In the implementation of QSDPNAL, we also develop efficient techniques for solving large scale linear systems of equations under certain subspace constraints. More specifically, simpler and yet better conditioned linear systems are carefully designed to replace the original linear systems and 
{novel}
shadow sequences are constructed to alleviate the numerical difficulties brought about by the crucial
subspace constraints. Extensive numerical results for various large scale QSDPs show that our two-phase 
{algorithm is highly efficient and robust}
in obtaining accurate solutions.
\end{abstract}

\section{Introduction}\label{intro}

Let $\cS_+^n$ and $\Sn_{++}$ be the {cones of  positive semidefinite and positive definite matrices}, respectively, in the space of $n\times n$ symmetric matrices $\cS^n$ endowed with the standard trace
inner product $\inprod{\cdot}{\cdot}$ and the Frobenius norm $\norm{\cdot}$. In this paper, we
consider the following convex quadratic semidefinite programming (QSDP) {problem}:
\begin{eqnarray}
 ({\bf P}) \quad  \min \left\{\frac{1}{2}\inprod{X}{\cQ X} + \inprod{C}{X}
 \mid \cA X=  b, \; X\in\cS_+^n \cap \cK \right\},
  \nonumber 
\end{eqnarray}
where $\cQ:\Sn\to \Sn $ is a self-adjoint positive semidefinite linear operator,  $\cA:\Sn \rightarrow \Re^{m}$ is a linear map whose adjoint is denoted as $\cA^*$,  $C\in \Sn$,  $b \in \Re^{m}$ are given data,  $\cK$ is a simple nonempty closed convex polyhedral set in $\Sn$,
e.g., $\cK =\{X \in\Sn \mid\, L\leq X\leq U\}$ with $L,U\in \Sn$ being given matrices.
The main objective of this paper is to design and analyse efficient algorithms for solving
({\bf P}) and its dual.
We are particularly
interested in the case where the dimensions $n$ and/or $m$ are large, and
{it may be impossible to explicitly store or compute
the matrix representation of $\cQ$.}
For example, if $\cQ = H\otimes H$ is the Kronecker product of a dense
matrix $H\in \cS^n_+$ with itself, then it would be
extremely expensive to store the matrix representation of $\cQ$ explicitly when
$n$ is larger than, say, 500.
As far as we are aware of,
the best solvers currently available for solving ({\bf P}) are based on inexact primal-dual interior-point methods \cite{toh2008inexact}. However, they are highly inefficient for solving large scale problems as
interior-point methods have severe inherent
ill-conditioning limitations which would make the convergence of a Krylov subspace iterative solver employed to compute the search directions
 to be extremely slow.
{While sophisticated preconditioners have been constructed in \cite{toh2008inexact} to alleviate the
ill-conditioning, the improvement is however not dramatic enough for the algorithm to handle large scale problems
comfortably.}
On the other hand, an interior-point method which employs a direct solver to compute the
search directions is  prohibitively expensive for solving ({\bf P}) since the cost is at least
$O( (m+n^2)^3)$ arithmetic operations per iteration.
{It is safe to say that
 there is currently no solver  which can efficiently handle large scale QSDP problem of the form ({\bf P}) and
our paper precisely aims to provide an efficient and robust solver for ({\bf P}).}


The algorithms which we will design later are based on the augmented Lagrangian
function for the dual of  ({\bf P}) (in its equivalent minimization form):
\begin{equation}
   ({\bf D}) \quad    \min \left\{\delta_{\cK}^*(-Z) +\frac{1}{2}\inprod{W}{\cQ W} - \inprod{b}{y}
\; \Big|
 \begin{array}{l}
 Z - \cQ W  + S + \cA^* y = C,
\\[3pt]
 S\in\Sn_+,\;
 W\in\cW,\; y\in\Re^m, Z \in \Sn
\end{array}
\right\},
\nonumber
\end{equation}
where
$\cW$ is any subspace of $\Sn$ containing the range space of $\cQ$
{(denoted as
$\Range(\cQ))$,
   $\delta_{\cK}^*(\cdot)$ is the  Fenchel conjugate of the
indicator function $\delta_{\cK}(\cdot)$.}

{
Due to its  great potential in applications and mathematical elegance, QSDP has been  studied quite
actively
both from the theoretical and numerical aspects  \cite{alfakih1999solving,
  higham2002computing, jiang2014partial,
krislock2004local, Nie2001, qi2006quadratically, toh2007inexact, toh2008inexact}.
}
For the recent theoretical developments, one may refer to  \cite{cui2016on,HanSZ2015,qi2009local,sun2010a} and references therein.
Here we focus on the numerical aspect and
we will next briefly review some of the methods available for solving QSDP problems. Toh et al \cite{toh2007inexact} and Toh \cite{toh2008inexact} proposed inexact primal-dual path-following interior-point methods to solve the special class of convex QSDP
without the constraint in $\cK$.
In theory, these methods can be used to solve QSDP problems
with inequality constraints and constraint in $\cK$ by reformulating the
problems {into} the required standard form. However, as already mentioned, in practice
interior-point methods are not efficient for solving QSDP problems beyond
moderate scales either due to the extremely high computational cost
per iteration or the inherent ill-conditioning of the linear systems governing the
search directions.
 In  \cite{zhao2009semismooth},  Zhao designed a semismooth Newton-CG augmented Lagrangian (NAL) method and analyzed its convergence for solving the primal QSDP problem ({\bf P}). However, the NAL algorithm often encounters numerical difficulty
(due to singular or nearly singular generalized Hessian)
when the polyhedral set constraint $X\in\cK$ is present.
Subsequently, Jiang et al \cite{jiang2012inexact} proposed an inexact accelerated proximal gradient method for least squares semidefinite programming
{with only equality constraints where}
the objective function in ({\bf P}) is expressed explicitly
in the form of $\norm{\cB X - d}^2$ for some given linear map $\cB$.

More recently, inspired by the successes achieved in \cite{SunTY3c,YangST2015} for solving the linear SDP problems with nonnegative constraints, Li, Sun and Toh  \cite{LiSunToh_scb2014} proposed a first-order algorithm, {known as the} Schur complement based semi-proximal alternating direction method of multipliers (SCB-sPADMM), for solving the dual form ({\bf D}) of QSDP.
As far as we aware of, \cite{LiSunToh_scb2014} is the first paper to advocate using the dual approach for solving QSDP problems
{even though the dual problem ({\bf D}) looks a lot more complicated than
the primal problem ({\bf P}), especially with the presence of the subspace constraint involving $\cW$.}
By leveraging on the Schur complement based decomposition technique developed in \cite{LiSunToh_scb2014,LiThesis2014},
 Chen, Sun and Toh \cite{chen2015efficient} also employed the dual approach by proposing an efficient inexact ADMM-type first-order method
(which we name as SCB-isPADMM) for solving problem ({\bf D}).
{Promising numerical results have been obtained by the dual based first-order algorithms
in solving various classes of QSDP problems to moderate accuracy \cite{LiSunToh_scb2014, chen2015efficient}.
Naturally one may hope
to also relay on the ADMM scheme to 
compute
highly accurate solutions.
 However, as one will observe from the numerical experiments presented
later in
Section \ref{sec:comp-example},  ADMM-type
methods are incapable of finding accurate solutions for difficult QSDP problems
due to its slow local convergence or stagnation}. On the other hand, recent studies on the convergence rate of augmented Lagrangian methods (ALM) for solving convex semidefinite programming with multiple solutions \cite{cui2016on} show that comparing to ADMM-type  methods, the ALM {can enjoy a faster convergence rate (in fact asymptotically superlinear) under milder conditions.
These recent advances thus strongly indicate that one should be able to design a highly efficient algorithm
based on the ALM for ({\bf D})
 for solving QSDP problems to high accuracy.}
More specifically, we will propose a two-phase augmented Lagrangian based algorithm with Phase I to generate a reasonably good initial point to warm start the Phase II algorithm so as to {compute} accurate solutions efficiently.
We call this new method \QSDPNAL since it extends the ideas of SDPNAL \cite{SDPNAL} and SDPNAL+ \cite{YangST2015} for linear SDP problems to QSDP problems. Although
the aforementioned two-phase framework has already been demonstrated to be highly efficient for solving linear {SDP} problems \cite{YangST2015,SDPNAL},
{it remains to be seen whether we can achieve equally or even more impressive performance
on various QSDP problems.}

{In recent years, it has become fashionable to design first-order
algorithms for solving convex optimization problems, with some even claiming their efficacy
in solving various challenging classes of matrix conic optimization problems based on limited performance evaluations.
However, based on our extensive numerical experience in solving large scale linear SDPs  \cite{SunTY3c,YangST2015,SDPNAL}, we have observed that
while first-order methods can be rather effective in solving easy problems which are well-posed and
nondegenerate, 
they
are typically powerless in solving difficult instances which are ill-posed or degenerate.
Even for a well designed  first-order algorithm with guaranteed convergence and highly optimized implementations,
such as the ADMM+ algorithm in \cite{SunTY3c}, a first-order method may still fail on slightly more challenging
problems. For example, the ADMM+ algorithm designed in \cite{YangST2015} can encounter varying degrees of difficulties
in solving linear SDPs arising from rank-one tensor approximation problems.
On the other hand, the  SDPNAL algorithm in \cite{SDPNAL} (which exploits second-order information)
is able to solve those problems very efficiently to high accuracy.
We believe that in order to design an efficient and robust algorithm to solve the highly challenging class of
matrix conic optimization problems including QSDPs, one must fully
combine the advantages offered by both the first and second order algorithms, rather than just solely relying
on first-order algorithms even though they may  appear to be easier to implement.
}

{Next} we briefly describe our algorithm {\sc Qsdpnal}.
Let $\cZ = \Sn\times\cW\times\Sn\times\Re^m$.
Consider the following Lagrange function associated with ({\bf D}):
\begin{equation*}
  \label{fun:lag}
  l(Z,W,S,y;X) := \delta_{\cK}^*(-Z) + \frac{1}{2}\inprod{W}{\cQ W} +\delta_{\Sn_+}(S)- \inprod{b}{y}
+ \inprod{Z-\cQ W + S + \cA^*y - C}{X},
\end{equation*}
where $ (Z,W,S,y)\in\cZ$ and $X\in \Sn$.
For  a given positive scalar $\sigma$, the augmented Lagrangian function for  ({\bf D}) is defined by
\begin{equation}\label{eq-aug-d-intro}
\cL_{\sigma}(Z,W,S,y;X) :=  l(Z,W,S,y;X) +
 \frac{\sigma}{2}\norm{Z-\cQ W+S + \cA^*y-C}^2,\quad (Z,W,S,y)\in\cZ, \; X\in \Sn.
\end{equation}
The algorithm which we will adopt in {\sc Qsdpnal}-Phase I is a variant of the SCB-isPADMM algorithm developed
in \cite{chen2015efficient}.
In {\sc Qsdpnal}-Phase II,
we design an augmented Lagrangian method (ALM) for solving ({\bf D}) where the
inner subproblem in each iteration is solved via inexact semismooth Newton based algorithms.
Given $\sigma_0 >0$, $(Z^0,W^0,S^0,y^0,X^0)\in\cZ\times \Sn$,
the $(k+1)$th iteration of the ALM  consists of the following steps:
\begin{equation*}\label{intro-prox-aug}
  \begin{aligned}
  &(Z^{k+1},W^{k+1},S^{k+1},y^{k+1})  \approx \argmin \left\{
  \mathcal{L}_{\sigma_k}(Z,W,S,y; X^{k})
  \mid\, (Z,W,S,y) \in \cZ
  \right\}, \\[5pt]
  &X^{k+1} = X^k + \sigma_k(Z^{k+1} - \cQ W^{k+1} + S^{k+1} + \mathcal{A}^*y^{k+1} - C),
  \end{aligned}
\end{equation*}
where $\sigma_k\in(0,+\infty)$.
The first issue in the above ALM is the choice of the subspace $\cW$.
The {obvious 
choice
$\cW = \Sn$ 
can lead to various difficulties}
in the implementation of the above algorithm. For example, since $\cQ:\Sn \to \Sn$ is only assumed to be positive semidefinite, the Newton systems corresponding to the inner subproblems 
 {may be}
singular and the sequence $\{W^{k}\}$ generated by the ALM can be unbounded. 
{As a result,}
it will be extremely difficult to analyze the convergence of the inner algorithms for solving the {ALM subproblems.} The second issue is that one needs to design {easy-to-check} stopping criteria for the inner subproblems, and
 to ensure the fast convergence of the ALM  {under reasonable  conditions imposed on the QSDP problems.
Concerning the first issue, we propose to choose $\cW = \Range(\cQ)$, although such a choice also leads to
obstacles which we will overcome in Section 4.}
Indeed, by restricting
$W\in\Range(\cQ)$, the difficulties in analyzing the convergence and the superlinear (quadratic) convergence of the Newton-CG algorithm
are circumvented as the  {possibilities of} singularity and unboundedness are removed. 
 For the second issue, under the restriction {that}
$\cW = \Range(\cQ)$, thanks to the recent advances in \cite{cui2016on},
we are able to design checkable stopping criteria for solving the
inner subproblem inexactly while establishing the global convergence of the above ALM.
{Moreover, we are able to establish}
the R-(super)linear convergence rate of the KKT residual. At first glance, the restriction that $W\in\Range(\cQ)$ appears to introduce severe numerical difficulties when we {need to} a solve linear system under this restriction. Fortunately, by carefully examining our algorithm and 
{devising novel
numerical techniques, we are able to overcome these difficulties as we shall see
in Section \ref{sec:numerical-issues}.}
Our preliminary evaluation of \QSDPNAL has demonstrated that our algorithm is capable of solving large scale general QSDP problems of the form ({\bf P}) to high accuracy very efficiently {and robustly}. For example, we are able to solve an elementwise weighted nearest correlation
matrix estimation problem with matrix dimension $n=10,000$
in less than 11 hours to the relative accuracy smaller than $10^{-6}$ in the KKT residual.
{Such a numerical performance has not been attained in the past.}

{
As the readers may have already observed, even though our goal
in developing algorithms for solving convex optimization problems
such as ({\bf P}) and ({\bf D})
is to design those with desirable theoretical properties such as asymptotic superlinear convergence,
it is our belief that it is equally if not even more important for the algorithms designed to be
practically implementable and able to achieve realistic numerical efficiency.
It is obvious that our proposed two-phase augmented Lagrangian based algorithm for solving ({\bf P}) and ({\bf D})
is designed based on such a belief.
}

The remaining parts of this paper are organized as follows. The next section is devoted to our main algorithm {\sc Qsdpnal}, which is a two-phase augmented Lagrangian {based} algorithm whose Phase I is used to generate a reasonably good initial point
to warm-start the Phase \rm{II} algorithm so as to obtain accurate solutions efficiently. In Section \ref{sec:ABCD}, we propose to solve the inner minimization subproblems of the ALM method by  semismooth Newton based algorithms and study their global and local superlinear (quadratic) convergence.  In Section \ref{sec:numerical-issues},
we discuss  {critical numerical} issues concerning the efficient implementation of {\sc Qsdpnal}. In Section \ref{sec:LS}, we discuss the special case of applying \QSDPNAL to  solve   least squares semidefinite programming problems.
{The extension of \QSDPNAL for solving QSDP problems with unstructured inequality constraints is
discussed in Section 5.2.}
In Section \ref{sec:comp-example},  we conduct numerical experiments to
evaluate the performance of  \QSDPNAL in solving various
  QSDP problems and their  extensions. We conclude our paper in the final section.

Below we list  several notation and definitions to be use in the paper. For a given closed proper convex function $\theta:\cX \to (-\infty,\infty]$, where $\cX$ is a
finite-dimensional real inner product space,
 the Moreau-Yosida proximal mapping $\text{Prox}_\theta(x)$ for   $\theta$ at a
 point $x$  is defined by
\begin{equation*}
\text{Prox}_\theta(x) := \argmin_{y\in \cX} \Big\{\theta(y) + \frac{1}{2}\|y-x\|^2\Big\}.
\end{equation*}
We will often make use of the following identity:
\begin{eqnarray*}
  \text{Prox}_{t \theta}( x) + t \text{Prox}_{\theta^*/t}(x/t) = x,
\end{eqnarray*}
where $t > 0$ is a given parameter, and
$\theta^*: \cX \to(-\infty,\infty]$ is the conjugate function of $\theta$.
If $\theta$ is the indicator function of a given closed convex set $D\subseteq \cX$, then the Moreau-Yosida proximal mapping is in fact the metric projector over $D$, denoted by $\Pi_{D}(\cdot)$. For any $x\in\cX$,
we define ${\rm dist}(x,D): = \inf_{d\in D}\norm{x - d}$.
For any $X\in\Sn$, we use $\lambda_{\max}(X)$ and $\lambda_{\min}(X)$ to {denote} the largest and the smallest eigenvalue of $X$, respectively.
{Similar notation is used when $X$ is replaced by the linear operator $\cQ$.}
\section{A two-phase augmented Lagrangian method}\label{sec:2phase}

In this section, we shall present our two-phase algorithm \QSDPNAL for solving
the QSDP problems  ({\bf D}) and ({\bf P}).
For the convergence analysis  of Algorithm {\sc Qsdpnal}, we
{need to
make the following standard assumption for ({\bf P}). Such an assumption is
analogous to the Slater's condition in the context of nonlinear programming in $\Re^m$.}

\begin{assumption}
  \label{assump:slater}
  There exists {$\widehat X\in\Sn_{++} \cap {\rm ri}(\cK)$} such that
  \[\cA (\cT_{\cK}(\widehat X)) = \Re^m,  \]
  where ${\rm ri}(\cK)$ denotes the relative interior of $\cK$ and $\cT_{\cK}(\widehat X)$ is the tangent cone of $\cK$ at point $\widehat X$.
\end{assumption}

\subsection{Phase I: An SCB based inexact semi-proximal ADMM}
In Phase I, we propose a new variant of the Schur complement based inexact semi-proximal ADMM (SCB-isPADMM) developed in \cite{chen2015efficient} to solve ({\bf D}). 
Recall the augmented Lagrangian function associated with problem ({\bf D}) defined in \eqref{eq-aug-d-intro}.

The detail steps of our Phase I algorithm for solving ({\bf D}) are given as follows.

\bigskip
\centerline{\fbox{\parbox{\textwidth}{
			{\bf Algorithm} {\bf \sc{Qsdpnal}}-{Phase I}: {\bf An SCB based inexact semi-proximal ADMM for ({\bf D})}.
\\[5pt]
			Select an initial point   $(W^0,S^0,y^0,X^0)\in\Range(\cQ)\times\Sn_+\times\Re^{m}\times\Sn$ and
$-Z^0\in
\textup{dom}(\delta^*_\cK)$. Let $\{\varepsilon_k\}$ be a summable sequence of
nonnegative numbers, and $\sigma >0$, $\tau\in(0,\infty)$ are given  parameters.
For $k=0,1,\ldots$, perform the following steps in each iteration.
			\begin{description}				
				\item[Step 1.]
Compute
\begin{align}
              \widehat W^{k} ={}& \argmin \{ \cL_{\sigma}(Z^k,W,S^k,y^k;X^k) - \inprod{\hat\delta_Q^k}{W} \mid  {W\in\Range(\cQ)} \}, \label{W1}
\\[5pt]
               Z^{k+1} ={}& \argmin \{ \cL_{\sigma}(Z,\widehat W^k,S^k, y^k;X^k) \mid Z\in \cS^n\},\nn
\\[5pt]
               W^{k+1} ={}& \argmin \{ \cL_{\sigma}(Z^{k+1},W,S^k, y^k;X^k) - \inprod{\delta_{\cQ}^k}{W} \mid {W\in\Range(\cQ)} \}, \label{W2}
\\[5pt]
              \hat y^{k} ={}& \argmin \{ \cL_{\sigma}(Z^{k+1},W^{k+1},S^k,y;X^k) - \inprod{\hat\delta_y^k}{y} \mid y\in \Re^m\},\nn
\\[5pt]
              S^{k+1} = {}& \argmin \{ \cL_{\sigma}(Z^{k+1},W^{k+1},S,\hat y^k; X^k) \mid
 S \in \cS^n\},\nn
\\[5pt]
              y^{k+1} ={}& \argmin \{ \cL_{\sigma}(Z^{k+1},W^{k+1},S^{k+1},y;X^k) - \inprod{\delta_y^k}{y} \mid y \in \Re^m\},\nn
              \end{align}
where
$\delta_y^k, \, \hat{\delta}_y^k \in \Re^{m}$, $\delta_{\cQ}^k,\,\hat\delta_{\cQ}^k \in\Range(\cQ)$ are error vectors such that
				\begin{equation*}
				\max \{ \norm{\delta_y^k}, \norm{\hat \delta_y^k}, \norm{\delta_{\cQ}^k},
 \norm{\hat\delta_{\cQ}^k}\}\leq \varepsilon_k.
				\end{equation*}
				\item [Step 2.] Compute
				$X^{k+1} = X^k + \tau\sigma(Z^{k+1} - QW^{k+1} +S^{k+1}+ \cA^*y^{k+1} -C).$						
				\end{description}
				}}}

\bigskip

\begin{remark}
  \label{rmk:error_terms}
  We shall explain here the roles of the error vectors
  $\delta_y^k, \, \hat{\delta}_y^k, \delta_{\cQ}^k$ and $\hat\delta_{\cQ}^k$. There is no need to choose these error vectors in advance. The presence of these error vectors simply indicates that the corresponding subproblems can be solved inexactly.
  For example, the updating rule of $y^{k+1}$ in the above algorithm can be interpreted as follows:
{find $y^{k+1}$ inexactly} through
  \[y^{k+1}\approx \argmin \cL_{\sigma}(Z^{k+1},W^{k+1},S^{k+1},y;X^k)\]
  such that the residual
  \[\norm{\delta_y^k} = \norm{b - \cA X^k - \sigma \cA(Z^{k+1} - \cQ W^{k+1} + S^{k+1} + \cA^* y^{k+1} - C)} \le \varepsilon_k.\]
\end{remark}
 {\begin{remark}
  \label{rmk:qsdpnal_1}
   In contrast to  Aglorithm SCB-isPADMM in \cite{chen2015efficient}, our Algorithm {\rm {\sc Qsdpnal}-Phase I} requires the subspace constraint $W\in\Range(\cQ)$ explicitly  in the subproblems \eqref{W1} and \eqref{W2}.
Note that
  due to the presence of the subspace constraint $W\in\Range(\cQ)$, there is no need to add extra proximal terms
{in the subproblems} corresponding to $W$ 
{to satisfy
the positive definiteness requirement needed in applying} the inexact Schur compliment based decomposition technique developed in \cite{LiSunToh_scb2014,LiThesis2014}.
 This is {certainly more elegant than} the indirect reformulation strategy considered in \cite{LiSunToh_scb2014,chen2015efficient}.
\end{remark}}
The convergence of the above algorithm follows from
\cite[Theorem 1]{chen2015efficient} without much difficulty, {and its proof is omitted.}
\begin{theorem}
  \label{thm:sGS-qsdp}
  Suppose that the solution set of {\rm({\bf P})} is nonempty and Assumption \ref{assump:slater} holds. Let $\{(Z^k,W^k,S^k,y^k,X^k)\}$ be the sequence generated by Algorithm {\rm {\sc Qsdpnal}-Phase I}. If $\tau\in(0,(1+\sqrt{5}\,)/2)$, then the sequence $\{(Z^k,W^k,S^k,y^k)\}$ converges to an optimal solution of {\rm ({\bf D})} and $\{X^k\}$ converges to an optimal solution of {\rm ({\bf P})}.
\end{theorem}

\begin{remark}
  \label{rmk:convergence_rate_ADMM}
Under some error bound conditions on the limit point of $\{(Z^k,W^k,S^k,y^k,X^k)\}$, one can derive the linear rate
of convergence of the exact version of Algorithm {\rm {\sc Qsdpnal}-Phase I}. For a recent study on this {topic}, see \cite{HanSZ2015} and the references therein. Here we will not address this issue as our Phase II
{algorithm} enjoys  a better rate of convergence   under weaker conditions.
\end{remark}

\subsection{Phase II: An augmented Lagrangian algorithm}

In this section, we discuss our Phase II algorithm for solving the {dual problem}
({\bf D}). The purpose of this phase is to obtain high accuracy solutions efficiently after warm-started by our Phase I algorithm.
Our Phase II algorithm has the following template.

\bigskip
\centerline{\fbox{\parbox{\textwidth}{
{\bf Algorithm} {\sc Qsdpnal}-Phase II: {\bf An augmented Lagrangian method of multipliers
for solving ({\bf D}).}
\\[5pt]
Let $\sigma_0 >0$ be a given parameter.
Choose $(W^0,S^0,y^0,X^0)\in\Range(\cQ)\times\Sn_+\times\Re^{m}\times\Sn$
and $-Z^0\in\textup{dom}(\delta^*_\cK)$.
For $k=0,1,\ldots$, perform the following steps in each iteration.
\begin{description}
  \item [Step 1.] Compute
  \begin{equation} \label{p2:palm-sub}
  \begin{aligned}
  &(Z^{k+1},W^{k+1},S^{k+1},y^{k+1})  \approx \argmin \left\{
 \begin{aligned}
  &\Psi_k(Z,W,S,y):=\mathcal{L}_{\sigma_k}(Z,W,S,y; X^{k}) \\[5pt]
  &  \mid\, (Z,W,S,y)\in\Sn\times\Range(\cQ)\times\Sn\times\Re^m
  \end{aligned}
  \right\}. \\[5pt]
  \end{aligned}
  \end{equation}
\item[Step 2.] Compute
     \begin{equation*}\label{eq:qsdpnal2_X}
     X^{k+1} = X^k + \sigma_k(Z^{k+1} - \cQ W^{k+1} + S^{k+1} + \mathcal{A}^*y^{k+1} - C).
     \end{equation*}
  Update $\sigma_{k+1} \uparrow \sigma_\infty\leq \infty$.
\end{description}
}}}
\bigskip

As an important issue on the implementation of the above algorithm, the stopping criteria for approximately solving subproblem \eqref{p2:palm-sub} shall be discussed here.
{Let  the feasible set for ({\bf P}) be denoted as $\cF:=\{X\in\Sn \mid \cA X = b, \, X\in\Sn_+\cap\cK \}$.
Define the feasibility residual function $\gamma:\Sn\to \Re$ for the primal problem ({\bf P}) by}
\begin{equation*}
  \label{eq:residual_F}
  \gamma(X) : =  \norm{b - \cA X} + \norm{X - \Pi_{\Sn_+}(X)} + \norm{X - \Pi_{\cK}(X)} ,\quad \forall\, X\in\Sn.
\end{equation*}
Note that $\gamma(X) = 0$ if and only if $X\in\cF$. Indeed, for $X\not\in \cF$,  $\gamma(X)$ provides an
{easy-to-compute measure on the primal infeasibility of $X$. Similar to}  \cite[Proposition 4.2]{cui2016on}, we can use this feasibility measure function to derive an upper bound on the distance of a given point to the feasible set $\cF$ in the next lemma. Its proof can be obtained without much difficulty by applying Hoffman's error bound \cite[Lemma 3.2.3]{facchinei2003finite} to the nonempty polyhedral convex set $\{X\in\Sn\mid \cA X = b,\, X\in\cK\}$, e.g., see
\cite[Theorem 7]{Bauschke1999Strong}.
\begin{lemma}\label{lemma:distXtoPFX}
 Assume that $\cF\cap\Sn_{++}\neq \emptyset$.
 Then, there exists a constant $\mu>0  $
such that
\begin{equation}\label{eq:lemmaXtoPFX}
  \norm{X - \Pi_{\cF}(X)} \le \mu (1+\norm{X}) \gamma(X), \quad \forall\,  X\in\Sn .
\end{equation}
\end{lemma}

 When the ALM is applied to solve ({\bf D}), numerically it is difficult to execute criteria $(A'')$ and $({\rm B}_1'')$ proposed in \cite{rockafellar1976augmented}. Fortunately, Lemma \ref{lemma:distXtoPFX} and recent advances in the analysis of the ALM
\cite{cui2016on} allow us to design easy-to-verify stopping criteria for the subproblems  in {\sc Qsdpnal}-Phase II.
For any $k\ge 0$, denote
\[f_k(X): =  -  \frac{1}{2}\inprod{X}{\cQ X} - \inprod{C}{X} - \frac{1}{2\sigma_k}\norm{X - X^k}^2 , \quad \forall\, X\in\Sn.\]
{Note that $f_k(\cdot)$ is in fact the objective function in the dual of problem \eqref{p2:palm-sub}.}
Let
$\{\varepsilon_k\}$ and $\{\delta_k\}$ be two given positive summable sequences.
Given $k\ge 0$ and $X^k\in\Sn$, we propose to {terminate}
the minimization of the subproblem \eqref{p2:palm-sub} in the $(k+1)$th iteration of Algorithm {\sc Qsdpnal}-Phase II with {either one of the following two easy-to-check} stopping criteria:
\begin{equation*}
\label{ALM_stop}
\begin{aligned}
&(\textup{A}) \quad
\left\{
\begin{aligned}
&\Psi_k(Z^{k+1},W^{k+1},S^{k+1},y^{k+1}) - f_k(X^{k+1})
\le \varepsilon_k^2/2\sigma_k, \\[5pt]
& (1+\norm{X^{k+1}})\gamma(X^{k+1})  \le \alpha_k \varepsilon_k/\sqrt{2\sigma_k},
\end{aligned}
\right.
\\[8pt]
&({\rm B})\quad
\left\{
\begin{aligned}
& \Psi_k(Z^{k+1},W^{k+1},S^{k+1},y^{k+1}) - f_k(X^{k+1}) \le
\delta_k^2 \norm{X^{k+1} - X^k}^2/2\sigma_k, \\[5pt]
&  (1+\norm{X^{k+1}})\gamma(X^{k+1})  \le \beta_k\delta_k\norm{X^{k+1} - X^k} /\sqrt{2\sigma_k},
\end{aligned}
\right.
\end{aligned}
\end{equation*}
where 
\[\alpha_k = \min\left\{1,\sqrt{\sigma_k},\frac{\varepsilon_k}{\sqrt{2\sigma_k}\norm{\nabla f_k(X^{k+1})}}\right\}\quad {\rm and} \quad \beta_k = \min\left\{1,\sqrt{\sigma_k},\frac{\delta_k\norm{X^{k+1} - X^k}}{\sqrt{2\sigma_k}\norm{\nabla f_k(X^{k+1})}}\right\}.
\]
\begin{lemma}
	\label{lemma:stopp_cond}
	Assume that Assumption \ref{assump:slater} holds.  Let $\mu$ be the constant given in \eqref{eq:lemmaXtoPFX}. Suppose that   for some $k\ge 0$, $X^k$ is not an optimal solution to problem ({\bf P}). Then one can always find
	$(Z^{k+1},W^{k+1},S^{k+1},y^{k+1})$  and $X^{k+1} = X^k + \sigma_k(Z^{k+1} - \cQ W^{k+1} + S^{k+1} + \cA^*y^{k+1} - C)$ satisfying  both   (A) and (B). Moreover,    (A) implies
{that}
	\[
	\Psi_k(Z^{k+1},W^{k+1},S^{k+1},y^{k+1}) - \inf \Psi_k
	\le \nu \varepsilon_k^2/2\sigma_k
	\]
	{}and (B) implies that
	\[
	\Psi_k(Z^{k+1},W^{k+1},S^{k+1},y^{k+1}) - \inf \Psi_k\leq (\nu
	\delta_k^2/2\sigma_k)\norm{X^{k+1} - X^k}^2,\]
	{}respectively, where
	\begin{equation}
	\label{eq:nu}
	\nu = 1 + \mu + \frac{1}{2}\lambda_{\max}(\cQ) + \frac{1}{2}\mu^2.
	\end{equation}
\end{lemma}

\begin{proof}
With the help of Lemma \ref{lemma:distXtoPFX}, one can establish the assertion in the same fashion as
in \cite[Proposition 4.2, Proposition 4.3]{cui2016on}.
\end{proof}

{For the subsequent analysis, we need to define  the essential objective function of ({\bf P}), which is given by}
\begin{equation*}
  \label{fun:essential_objP}
  \begin{aligned}
  \phi(X) := {}& - \inf \, \{\,l(Z,W,S,y;X)\mid (Z,W,S,y)\in\Sn\times\Range(\cQ)\times\Sn\times\Re^m\} \\[5pt]
  = {}&\left\{
  \begin{aligned}
  & \frac{1}{2}\inprod{X}{\cQ X} + \inprod{X}{C}+ \delta_{\Sn_+}(X) + \delta_{\cK}(X) \quad \mbox{if}\  \cA X = b , \\[5pt]
  & +\infty \quad \mbox{otherwise}.
  \end{aligned}
  \right.
  \end{aligned}
\end{equation*}
{For convenience, we also let $\Omega = \partial\phi^{-1}(0)$ to denote the solution set of ({\bf P}). }

We say that for ({\bf P}), the second order growth condition holds at an optimal solution
{$\overline X\in \Omega$ with respect to the set
$\Omega$} if there exist $\kappa>0 $ and a neighborhood $U$ of $\overline X$ such that
\begin{equation}\label{eq:second_order_growth}
\phi(X) \ge \phi(\overline X)  + {\kappa^{-1} {\rm dist}^2(X,\Omega)}, \quad \forall\, X\in U.
\end{equation}
Let the objective function $g:\Sn \times\Range(\cQ)\times\Sn \times \Re^m\to (-\infty,+\infty]$ associated with ({\bf D}) be given as follows:
\[g(Z,W,S,y) := \delta_{\cK}^*(-Z) + \frac{1}{2}\inprod{W}{\cQ W} + \delta_{\Sn_+}(S) - \inprod{b}{y}, \quad \forall \, (Z,W,S,y)\in \Sn \times\Range(\cQ)\times\Sn \times \Re^m.\]
Now, with Lemma \ref{lemma:stopp_cond},  we can prove  the global and local (super)linear convergence of Algorithm {\sc Qsdpnal}-Phase II
{by adapting the proofs in}
 \cite[Theorem 4]{rockafellar1976augmented}
and \cite[Theorem 4.2]{cui2016on}.
It shows that, for most QSDP problems, one can always expect the KKT residual of the sequence generated by {\sc Qsdpnal}-Phase II {to converge} at least R-(super)linearly.

\begin{theorem}
  \label{thm:qsdpnal-2-global}
   Suppose that {$\Omega$,
   the solution set of {\rm({\bf P})}, is nonempty}
   and Assumption \ref{assump:slater} holds.
Then the sequence $\{(Z^k,W^k,S^k,y^k,X^k)\}$ generated by Algorithm {\rm {\sc Qsdpnal}-Phase II} under the stopping criterion $(\textup{A})$
for all $k\ge0$ is bounded, and $\{X^k\}$ converges to $X^{\infty}$, an optimal solution {of} {\rm ({\bf P})}, and $\{(Z^k,W^k,S^k,y^k)\}$ converges to an optimal solution {of} {\rm ({\bf D})}. Moreover, for all $k\ge 0$, it holds that
\begin{equation*}
  \begin{aligned}
    &g(Z^{k+1},W^{k+1},S^{k+1},y^{k+1}) - \inf \,({\bf D})\\[5pt]
    \le{}
    &\Psi_k(Z^{k+1},W^{k+1},S^{k+1},y^{k+1}) - \inf \Psi_k + (1/2\sigma_k)(\norm{X^k}^2 - \norm{X^{k+1}}^2).
  \end{aligned}
\end{equation*}

Assume that for {\rm ({\bf P})}, the second order growth condition \eqref{eq:second_order_growth} holds at $X^{\infty}$ with respect to the set
{$\Omega$}, i.e., there exists a constant $\kappa > 0$ and a neighborhood $U$ of $X^{\infty}$ such that
\[\phi(X) \ge \phi(X^{\infty})  + \kappa^{-1} {\rm dist}^2(X, {\Omega}), \quad \forall\, X\in U.\]
{Suppose that
the algorithm is} executed under  {criteria {\rm (A)} and {\rm (B)}} for all $k\ge 0$ and $\nu$ is the constant given in \eqref{eq:nu}. Then, for all $k$ sufficiently large, it holds that
\begin{eqnarray}
     {\rm dist}(X^{k+1},  {\Omega}) \le \theta_k {\rm dist}(X^{k}, {\Omega}), \label{eq:asyP_Qsupl} \\[5pt]
     \norm{Z^{k+1} - \cQ W^{k+1} + S^{k+1} + \cA^*y - C} \le \tau_k{\rm dist}(X^k,  {\Omega}), \label{eq:asyD_Rsupl}\\[5pt]
    g(Z^{k+1},W^{k+1},S^{k+1},y^{k+1}) - \inf \,({\bf D}) \le \tau_k'{\rm dist}(X^k,  {\Omega}), \label{eq:asygap_Rsupl}
\end{eqnarray}
where
\begin{equation*}
\begin{aligned}
  & 1>\theta_k = \big(\kappa/\sqrt{\kappa^2 + \sigma_k^2 }+ 2\nu\delta_k\big)(1 - \nu\delta_k)^{-1} \to \theta_{\infty} = \kappa/\sqrt{\kappa^2 + \sigma_{\infty}^2} \quad (\theta_{\infty} = 0 \; {\rm if}\; \sigma_{\infty} = \infty),\\[5pt]
  & \tau_k = \sigma_k^{-1}(1-\nu\delta_k)^{-1} \to \tau_{\infty} = 1/\sigma_{\infty}\quad (\tau_{\infty} = 0 \; {\rm if}\; \sigma_{\infty} = \infty), \\[5pt]
  & \tau_k' = \tau_k(\nu^2\delta_k^2\norm{X^{k+1} - X^k} + \norm{X^{k+1}} + \norm{X^k})/2 \to \tau'_{\infty} = \norm{X^{\infty}}/\sigma_{\infty}\quad
  (\tau'_{\infty} = 0 \; {\rm if}\; \sigma_{\infty} = \infty).
\end{aligned}
\end{equation*}
\end{theorem}

\bigskip
{Next
 we give a few} comments on the convergence rates and assumptions made in Theorem \ref{thm:qsdpnal-2-global}.
\begin{remark}
  \label{rmk: convergence_rate_KKT_res}
  Under the assumptions of Theorem \ref{thm:qsdpnal-2-global}, we have proven that the KKT residual, corresponding to ({\bf P}) and ({\bf D}), along the sequence $\{(Z^{k},W^k,S^k,y^k,X^k)\}$
  converges at least R-(super)linearly. Indeed, {under stopping criteria {(A)}, {(B)} and} from \eqref{eq:asyP_Qsupl},\eqref{eq:asyD_Rsupl} and \eqref{eq:asygap_Rsupl}, we know that the primal feasibility, the dual feasibility and the duality gap all converge at least R-(super)linearly.
\end{remark}

\begin{remark}
  \label{rmk:mild_second_order_growth}
  The assumption that the second order growth condition \eqref{eq:second_order_growth} holds for ({\bf P}) is quite mild. Indeed, it holds when any optimal solution $\overline X$ of ({\bf P}), together with any of its multiplier $\overline S \in \Sn_+$ corresponding only to the semidefinite constraint, satisfies the strict complementarity condition \cite[Corollary 3.1]{cui2016on}. It is also valid when the ``no-gap'' second order sufficient condition holds at the optimal solution\footnote{In this case, the optimal solution set to ({\bf P}) is necessarily a singleton though ({\bf D}) may have multiple solutions.} to ({\bf P}) \cite[Theorem 3.137]{Bonnans2000perturbation}.
\end{remark}

\section{Inexact semismooth Newton based algorithms for solving the inner subproblems  \eqref{p2:palm-sub} in ALM}\label{sec:ABCD}

 In this section, we will design efficient inexact semismooth Newton based algorithms
 to solve the inner subproblems \eqref{p2:palm-sub}
 in the augmented Lagrangian method, where each subproblem takes
the form of:
\begin{equation}\label{prob-abcd}
\min\left\{
\Psi(Z,W,S,y):= \cL_{\sigma}(Z,W,S,y;\widehat X)
\mid\, (Z,W,S,y)\in\Sn\times\Range(\cQ)\times\Sn\times\Re^{m}
\right\}
\end{equation}
for a given $ \widehat X \in \Sn$. Note that the dual problem of \eqref{prob-abcd} is given as follows:
\begin{equation*}\label{prob:abdc_D}
\max \left\{ - \frac{1}{2}\inprod{X}{\cQ X} - \inprod{C}{X} - \frac{1}{2\sigma}\norm{X - \widehat X}^2
 \mid \cA X=  b, \; X\in\cS_+^n,\; X \in \cK \right\}.
\end{equation*}

Under Assumption \ref{assump:slater}, from \cite[Theorems 17 \& 18]{rockafellar1974conjugate}, we know that the optimal solution set of problem \eqref{prob-abcd} is nonempty and for any
$\alpha \in \Re$, the level set $\cL_{\alpha}:=\{(Z,W,S,y)
\in \Sn\times\Range(\cQ)\times\Sn\times\Re^{m} \,\mid\,
\Psi(Z,W,S,y)\le \alpha\}$ is a closed and bounded convex set.

\subsection{A semismooth Newton-CG algorithm for \eqref{prob-abcd} with $\cK = \Sn$}
\label{subsec-sncg}

Note that {in quite a number of applications,
the polyhedral convex set $\cK$ is actually the whole space $\Sn$}. Therefore, we shall first study how the inner problems \eqref{prob-abcd} in Algorithm ALM can be solved efficiently when $\cK = \Sn$.
Under this setting, $Z$ is vacuous, i.e., {$Z=0$.}

Let  $\sigma >0$ be given.
Denote
\[S(W,y) := \cA^*y - \cQ W -  \widehat C, \quad \forall \, (W,y)\in \Range(\cQ) \times \Re^m.\]
where
 $\widehat C = C - \sigma^{-1}\widehat X$.
Observe that if
\begin{equation*}
  \label{sncg-prob}
  (W^*,S^*,y^*) = \argmin\{\Psi(0,W,S,y)  \,\mid\, (W,S,y)\in\Range(\cQ)\times\Sn\times\Re^m\},
\end{equation*}
then $(W^*,S^*,y^*)$ can be computed in the following manner
\begin{eqnarray}
  &&(W^*,y^*) = \argmin \left\{
  \begin{aligned}
  \varphi(W,y)
\,|\, (W,y)\in\Range(\cQ)\times\Re^m
  \end{aligned}
  \right\}, \label{eq-wy}\\[5pt]
  &&S^* = \Pi_{\Sn_+}(-S(W^*,y^*)), \nonumber
\end{eqnarray}
where
\[\varphi(W,y):= \frac{1}{2} \inprod{W}{\cQ W} - \inprod{b}{y}
  + \frac{\sigma}{2}\norm{\Pi_{\Sn_+}(S(W,y))}^2, \quad
   \forall\, (W,y)\in\Range(\cQ)\times\Re^m.\]
Note that $\varphi(\cdot,\cdot)$ is a continuously differentiable function on $\Range(\cQ)\times \Re^m$ with
\begin{equation*}
  \nabla\varphi(W,y) = \left(\begin{array}{l}
     \cQ W - \sigma\cQ \Pi_{\Sn_+}(S(W,y))
\\[5pt]
    -b +  \sigma \cA \Pi_{\Sn_+}(S(W,y))
  \end{array} \right).
\end{equation*}
Then, solving \eqref{eq-wy} is equivalent to solving the following nonsmooth equation:
\begin{equation*}\label{eq-wy-nonsmooth}
\nabla\varphi(W,y) = 0, \quad (W,y)\in\Range(\cQ)\times\Re^m.
\end{equation*}
Since $\Pi_{\Sn_+}$ is strongly semismooth \cite{SunS2002},
we can design a semismooth Newton-CG (SNCG) method to solve \eqref{eq-wy} and could expect to get a fast superlinear or even quadratic convergence. For any $(W,y)\in\Range(\cQ)\times\Re^m$, define
\[\hat\partial^2 \varphi(W,y) := \left[ \begin{array}{cc}
        \cQ &   \\
           & 0
   \end{array} \right]+  \sigma \left[ \begin{array}{c}
       \cQ  \\
       -\cA
   \end{array} \right] \partial\Pi_{\Sn_+}(S(W,y)) [\cQ \; -\cA^*], \]
where $\partial\Pi_{\Sn_+}(S(W,y))$ is the Clarke subdifferential \cite{Clarke83} of $\Pi_{\Sn_+}(\cdot)$ at $S(W,y)$. Note that from \cite{hiriart1984generalized}, we know that
\begin{equation*}
\hat{\partial}^2 \varphi(W,y)\, (d_W,d_y) = {\partial}^2 \varphi(W,y)\, (d_W,d_y), \quad  \forall \, (d_W,d_y) \in \Range(\cQ)\times\Re^m,
\label{eq-Clarke}
\end{equation*}
where ${\partial}^2 \varphi(W,y)$ denotes the generalized Hessian of $\varphi$ at $(W,y)$, i.e., the Clarke subdifferential of $\nabla \varphi$ at $(W,y)$.

Given $(\widetilde W, \tilde y)\in\Range(Q)\times\Re^m$,
consider the following eigenvalue decomposition:
\begin{eqnarray*}\label{decomp-M}
S(\widetilde W, \tilde y) = \cA^*\tilde y - \cQ \widetilde W -  \widehat C   = P \, \Gamma \,P^{\T},
\end{eqnarray*}
where $P\in\Re^{n\times n}$ is an orthogonal matrix whose columns are
eigenvectors, and $\Gamma$ is the corresponding diagonal matrix of eigenvalues,
arranged in a nonincreasing order:
$\lambda_1 \geq \lambda_2 \geq
\cdots \geq \lambda_n$. Define the following  index sets
\begin{eqnarray*}
\alpha := \{i \mid \lambda_i>0\}, \quad  \bar{\alpha} :=\{i \mid
\lambda_i\leq 0\}.
\end{eqnarray*} 

\noindent
We define the operator $U^0 : \mathcal{S}^n \rightarrow \mathcal{S}^n$ by
\begin{eqnarray} \label{def-W0}
    U^0 (H) := P(\Sigma \circ (P^{\T}H P))P^{\T}, \quad H \in \mathcal{S}^n,
\end{eqnarray}
where $``\circ"$ denotes the Hadamard product of two matrices,
\begin{eqnarray*}\label{def-nu}
\quad \Sigma = \left[
\begin{array}{cc}
E_{\alpha \alpha}  & \nu_{\alpha \bar{\alpha}}\\[4pt]
\nu^{\T}_{\alpha \bar{\alpha}} & 0
\end{array} \right], \quad
\nu_{ij} := \frac{\lambda_i}{\lambda_i-\lambda_j}, \,\, i \in
\alpha, j \in \bar{\alpha},
\end{eqnarray*}
and $E_{\alpha \alpha} \in \mathcal{S}^{|\alpha|}$ is the matrix of ones.
In  \cite[Lemma 11]{pang2003semismooth}, it is proved that
$$
U^0 \in\partial \Pi_{\Sn_+}(S(\widetilde W, \tilde y)).
$$
Define
\begin{equation}\label{p2:eq-netwon-partial}
V^0 := \left[ \begin{array}{cc}
        \cQ &   \\
           & 0
   \end{array} \right]+  \sigma  \left[ \begin{array}{c}
       \cQ  \\
       -\cA
   \end{array} \right]  U^0 [\cQ \, -\cA^*].
 \end{equation}
 Then, we have $V^0\in\hat\partial^2\varphi(\widetilde W,\tilde y)$.

 After all the above preparations, we can design the following
semismooth Newton-CG method as in \cite{SDPNAL} to solve \eqref{eq-wy}.

\bigskip
\noindent

\centerline{\fbox{\parbox{\textwidth}{
{\bf Algorithm SNCG}: {\bf A semismooth Newton-CG algorithm.}
\\[5pt]
Given $\mu \in (0, 1/2)$, $\bar{\eta} \in (0, 1)$, $\tau \in (0,1]$, $\tau_1,\tau_2\in(0,1)$ and $\delta \in (0, 1)$. Choose $(W^0,y^0)\in\Range(\cQ)\times\Re^m$. Set $j=0$. Iterate the following steps.
\begin{description}
\item[Step 1.]  Choose $U_j\in \partial\Pi_{\Sn_+}(S(W^j,y^j))$ defined as in \eqref{def-W0}.  Let $V_j$ be given in \eqref{p2:eq-netwon-partial} with $U^0$ replacing by $U_j$ and $\epsilon_j = \tau_1\min\{\tau_2,\norm{\nabla \varphi(W^j,y^j)}\}$. Apply the CG algorithm to find an approximate solution $(d_W^j,d_y^j)\in\Range(\cQ)\times\Re^m$ to
                       \begin{eqnarray}\label{eqn-epsk}
                         V_j (d_W,d_y) + \epsilon_j (0,d_y) = -\nabla \varphi(W^j,y^j)
                       \end{eqnarray}                       such that
\begin{equation*}
\norm{V_j(d_W^j,d_y^j)   + \epsilon_j (0,d_y^j)+\nabla \varphi(W^j,y^j)}\le \eta_j := \min(\bar{\eta}, \| \nabla \varphi(W^j, y^j)\|^{1+\tau}).
\label{eq-eta}
\end{equation*}
\item[Step 2.]  Set $\alpha_j = \delta^{m_j}$, where $m_j$ is the first nonnegative integer $m$ for which
                         \begin{equation*}\label{Armijo}
                          \varphi(W^j + \delta^{m} d_W^j,y^j+\delta^m d_y^j) \leq \varphi(W^j,y^j) + \mu \delta^{m}
                           \langle \nabla \varphi(W^j,y^j), (d^j_W,d^j_y) \rangle.
                          \end{equation*}
\item[Step 3.] Set $W^{j+1} = W^j + \alpha_j \, d_W^j$ and
$y^{j+1} = y^j + \alpha_j \, d_y^j$.
\end{description}
}}}
\bigskip
\vskip 10 true pt

The convergence results for the above SNCG algorithm are stated
in the next theorem.
\begin{theorem}
Suppose that Assumption \ref{assump:slater} holds. Then Algorithm SNCG generates a bounded sequence $\{(W^j,y^j)\}$ and any accumulation point $(\overline W, \bar y) \in \Range(\cQ)\times\Re^m$ is an optimal solution to problem \eqref{eq-wy}.
\end{theorem}


The following proposition is the key ingredient  in our subsequent  convergence analysis.

\begin{prop}\label{prop:psd-RangeQ}
	Let $\cU:\Sn\to\Sn$ be a self-adjoint positive semidefinite linear operator and $\sigma>0$. Then, it holds that $\cA \cU \cA^*$ is positive definite  if and only if
	\begin{equation}\label{psd-Q}
	\Inprod{\left[\begin{array}{c}
		W\\ y
		\end{array}\right]}{\left(\left[ \begin{array}{cc}
		\cQ &   \\
		& 0
		\end{array} \right]+\sigma \left[ \begin{array}{c}
		\cQ  \\
		-\cA
		\end{array} \right] \cU [\cQ \, -\cA^*]\right)\left[\begin{array}{c}
		W\\
		y
		\end{array}\right]} > 0
	\end{equation}
	for all $(W,y)\in\Range(\cQ)\times\Re^m\backslash \{(0,0)\}.$
	
\end{prop}
\begin{proof} Since the ``if" statement obviously holds true,
	we only need to prove the ``only if" statement.
	Note that
	$$ \inprod{W}{\cQ W} > 0,\quad \forall\, W\in\Range(\cQ){\backslash\{0\}}.
	$$
	Now suppose that $\cA \cU \cA^*$ is positive definite, and hence  nonsingular.
	By the Schur complement condition for ensuring  the  positive definiteness of a linear operator, we know that \eqref{psd-Q} holds if and only if
	\begin{equation}\label{psd-QAschur}
	\inprod{W}{(\cQ + \sigma \cQ \cU\cQ - \sigma \cQ \cU\cA^*(\cA \cU\cA^*)^{-1}\cA \cU\cQ)W} > 0,\quad\forall\, W\in\Range(\cQ){\backslash\{0\}}.
	\end{equation}
	But for any $W\in\Range(\cQ) \backslash\{ 0\}$, { we have that $\inprod{W}{\cQ W} > 0$, and }
	\begin{eqnarray*}
		&& \hspace{-0.7cm}
		\inprod{W}{(\cQ \cU\cQ - \cQ \cU \cA^*(\cA \cU\cA^*)^{-1}\cA \cU \cQ)W}\; =\;
		\inprod{W}{\cQ \cU^{\frac{1}{2}}(\cI -\cU^{\frac{1}{2}}\cA^*(\cA \cU\cA^*)^{-1}\cA \cU^{\frac{1}{2}})\cU^{\frac{1}{2}}\cQ W}\\[5pt]
		& = &
		\inprod{ \cU^{\frac{1}{2}} \cQ W}{(\cI -\cU^{\frac{1}{2}}\cA^*(\cA \cU\cA^*)^{-1}\cA \cU^{\frac{1}{2}})\cU^{\frac{1}{2}}\cQ W}
		\;\geq \; 0.
	\end{eqnarray*}
	Hence, \eqref{psd-QAschur} holds automatically. This completes the proof of the proposition.
\end{proof}

Base on the above proposition, under the constraint nondegeneracy  condition for ({\bf P}), we shall show in the next theorem that one can still ensure the positive definiteness of the coefficient matrix in the semismooth Newton system at the solution point.

\begin{theorem}
	\label{SNCG-no-prox}
	Let $(\overline W, \by)$ be the optimal solution for problem \eqref{eq-wy}.
	Let $\overline Y := \Pi_{\Sn_+}(\cA^*\bar y - \cQ \overline W -   \widehat C)$.
	The following conditions are equivalent:
	\begin{enumerate}
		\item[{\rm (i)}]
		The constraint nondegeneracy  condition,
		\begin{equation}\label{eq:cons_nondegen}
		\cA\,{\rm lin}(\cT_{\Sn_+}(\overline{Y})) = \Re^m,
		\end{equation}
		holds at $\overline Y$, where ${\rm lin}(\cT_{\Sn_+}(\overline{Y}))$ denotes the lineality space
		of the tangent cone of $\Sn_+$ at $\overline{Y}$.
		\item[{\rm (ii)}] Every element in
		\begin{equation*}\label{psd-nondegen}
		\left[ \begin{array}{cc}
		\cQ &   \\
		& 0
		\end{array} \right]+\sigma \left[ \begin{array}{c}
		\cQ  \\
		-\cA
		\end{array} \right] \partial\Pi_{\Sn_+}(\cA^*\bar y - \cQ \overline W - \widehat C) [\cQ \, -\cA^*] \end{equation*}
		is self-adjoint and positive definite on $\Range(\cQ)\times\Re^m.$
	\end{enumerate}
\end{theorem}
\begin{proof}
	In the same fashion as in \cite[Proposition 3.2]{SDPNAL}, we can prove that $ \cA \cU \cA^* $ is positive definite for all $\cU \in  \partial\Pi_{\Sn_+}(\cA^*\bar y - \cQ \overline W - \widehat C)$ if only if (i) holds. Then, by Proposition \ref{prop:psd-RangeQ}, we readily obtain the desired results.
\end{proof}

\begin{theorem}\label{convergence-zwy-newton}
 Assume that Assumption \ref{assump:slater} holds. Let $(\overline W, \bar y)$ be an accumulation point of the infinite sequence $\{(W^j,y^j)\}$ generated by Algorithm SNCG for solving problem \eqref{eq-wy}.
 Assume that the constraint nondegeneracy condition \eqref{eq:cons_nondegen}
 holds at $\overline Y: = \Pi_{\Sn_+}(\cA^*\bar y - \cQ \overline W -  \widehat C)$. Then, the whole sequence $\{(W^j,y^j)\}$ converges to  $(\overline W, \bar y)$ and
\begin{equation*}
\|(W^{j+1},y^{j+1}) - (\overline W,\bar y) \| = O(\norm{(W^j,y^j) - (\overline W,\bar y)})^{1+\tau}.
\end{equation*}
\end{theorem}
\begin{proof}
From Theorem \ref{SNCG-no-prox}, we know that under the constraint nondegeneracy condition \eqref{eq:cons_nondegen}, every $V \in\hat\partial^2 \varphi(\overline W,\bar y)$ is self-adjoint and positive definite on $ \Range(\cQ)\times\Re^n$.
{Hence} one can obtain the desired results from
\cite[Theorem 3.5]{SDPNAL} by further noting the strong semismoothness of $\Pi_{\Sn_+}(\cdot)$.
\end{proof}

\subsection{Semismooth Newton based inexact ABCD algorithms for \eqref{prob-abcd} when $\cK\not=\Sn$}

 When $\cK\neq \Sn$, we will adapt the recently developed inexact accelerated block coordinate descent (ABCD)  algorithm \cite{ABCD} to solve the inner subproblems  \eqref{prob-abcd} in the augmented Lagrangian method.

The detailed steps of the ABCD algorithm to be used for solving \eqref{prob-abcd} will be presented below. In this algorithm, $(Z,W,S,y)$ is decomposed into two groups, namely
$Z$ and $(W,S,y)$. In this case, $(W,S,y)$ is regarded as a single block and the corresponding subproblem in the ABCD algorithm can only be solved by an iterative method inexactly. Here, we propose to
{develop} a semismooth Newton-CG method to solve the corresponding subproblem.


\bigskip
\centerline{\fbox{\parbox{\textwidth}{
			{\bf Algorithm ABCD($Z^{0},W^0,S^{0},y^{0},\widehat X,\sigma$)}: {\bf An inexact ABCD algorithm for \eqref{prob-abcd}.}
			\\[5pt]
Given  $(W^0,S^0,y^0)\in\Range(\cQ)\times\Sn_+\times\Re^{m}$,
$-Z^{0}\in\dom(\delta^*_\cK)$ and $\eta > 0$,
			set  $(\tZ^1,\tW^1, \tS^1,\ty^1) = (Z^{0},W^0,S^{0},y^{0})$ and
$t_1=1 $. Let $\{\varepsilon_l\}$ be a nonnegative summable sequence.
			For $l = 1,\ldots,$ perform the following steps in each iteration.
\begin{description}				
		\item[Step 1.]
Let $\tR^l = \sig (\tS^k + \cA^*\ty^l-\cQ\tW^l - C +\sig^{-1} \widehat{X})$.
Compute
	\begin{align}				
 Z^l = {}&\argmin \big\{\Psi(Z,\tW^{l},\tS^{l},\ty^l)\,\mid\, Z\in\Sn \big\} =
\frac{1}{\sig}\big( \Pi_{\cK}(\tR^l) - \tR^l\big),
\nn\\[5pt]
(W^l,S^l,y^{l}) ={}& \argmin
\left\{
\begin{aligned}
&\Psi(Z^l,W,S,y) + \frac{\eta}{2}\norm{y - \tilde y^l}^2 -\inprod{\delta_y^l}{y}
-\inprod{\delta_{\cQ}^l}{W}
\\[5pt]
&\mid\, (W,S,y)\in\Range(\cQ)\times\Sn\times\Re^m
\end{aligned}
\right\}, \label{sncg-abcd}					
\end{align}		
where
 $\delta_y^l \in \Re^{m}$,  $\delta_{\cQ}^l \in \Range(\cQ)$  are error vectors such that
$$
\max \{ \norm{\delta_y^l},  \norm{\delta_{\cQ}^l}\}\leq \varepsilon_l/t_l.
$$
\item [Step 2.] Set $t_{l+1} = \frac{1+\sqrt{1+4t_l^2}}{2}$, $\beta_l=\frac{t_l-1}{t_{l+1}}$.
               Compute
				\begin{equation*}
				\hspace{-0.7cm}
               	\tW^{l+1} = W^l + \beta_l (W^l - W^{l-1}), \; \tS^{l+1} = S^l +\beta_l (S^l-S^{l-1}), \;					
					\ty^{l+1} = y^l + \beta_l(y^l-y^{l-1}).	
				\end{equation*}					
			\end{description}
		}}}
\bigskip

{
Note that
in order to meet the convergence requirement of the inexact ABCD algorithm,
a proximal term involving
the positive parameter $\eta$ is added in \eqref{sncg-abcd} to ensure the strong convexity of the objective function in the subproblem. For the computational efficiency, one can always take $\eta$ to be a small number, say $10^{-6}$.
For the subproblem
\eqref{sncg-abcd}, it can be solved by a
semismooth Newton-CG algorithm similar to the one developed in Subsection \ref{subsec-sncg}. Since $\eta >0$, the superlinear convergence of  such a  semismooth Newton-CG algorithm 
can also be proven based on
the} strong semismoothness of $\Pi_{\Sn_+}(\cdot)$ and the symmetric positive definiteness of the corresponding generalized Hessian.

The convergence results for the above Algorithm ABCD are stated in
the next theorem, whose proof essentially follows
from {that in} \cite[Theorem 3.1]{ABCD}. Here, we omit the proof for brevity.

\begin{theorem}
  \label{ABCD-1}
 Suppose that Assumption \ref{assump:slater} holds and $\eta >0$.  Let $\{(Z^l,W^l,S^l,y^l)\}$ be the sequence generated by Algorithm ABCD. Then,
 \[\inf_{Z}\Psi(Z,W^l,S^l,y^l) - \Psi(Z^*,W^*,S^*,y^*) = O(1/l^2)\]
 where $(Z^*,W^*,S^*,y^*)$ is an optimal solution of problem \eqref{prob-abcd}. Moreover,
 the sequence $\{(Z^l,W^l,S^l,y^l)\}$ is bounded and {all of its cluster points are optimal solutions
 to problem \eqref{prob-abcd}.}
\end{theorem}
\section{Numerical issues in \QSDPNAL}
\label{sec:numerical-issues}
\def\hW{\widehat{W}}

In Algorithm {\sc Qsdpnal}-Phase I, in order to obtain $\widehat W^k$ and $W^{k+1}$
at the $k$th iteration, we need to solve the following linear system of equations
\begin{equation}
  \label{eq-qppal-w}
  (\cQ +\sigma \cQ^2)  W \approx \cQ R,\quad W\in\Range(\cQ)
\end{equation}
with the residual
\begin{equation}
  \label{eq-qppal-w-r}
   \norm{\cQ R- (\cQ+ \sigma \cQ^2) W}\le \varepsilon,
\end{equation}
where $R\in\Sn$ and $\varepsilon >0$ are given.
Note that the exact solution to
\eqref{eq-qppal-w}  is unique since $\cQ+\sig \cQ^2$ is positive definite
on $\Range(\cQ)$.
But the linear system is typically very large even for a moderate $n$, say $n= 500$.
Under the high dimensional setting which we are particularly interested in, the matrix representation of $\cQ$ is generally not available or too expensive to be stored explicitly.
Thus \eqref{eq-qppal-w} can only be solved inexactly by an iterative method. However when
$\cQ$ is singular (and hence
$\Range(\cQ)\neq \Sn$), due to the presence of the subspace constraint $W\in\Range(\cQ)$, it is extremely difficult
to apply preconditioning to
\eqref{eq-qppal-w} while ensuring that the approximate solution is
contained in  $\Range(\cQ)$.
Fortunately, as shown in the next proposition, instead of solving \eqref{eq-qppal-w} directly,
we can solve a simpler and yet better conditioned linear system to overcome this difficulty.

\begin{prop}\label{prop:p1-com}
  Let $\widehat W$ be an approximate solution to the following linear system:
\begin{equation}\label{eq-p1com-w}
(\cI + \sigma \cQ) W \approx R
\end{equation}
with the residual satisfying
\begin{equation*}\label{eq-p1-com-r}
\norm{R-(\cI+\sigma \cQ) \hW }\le \frac{\varepsilon}{{\lambda_{\max}(\cQ)}}.
\end{equation*}
Then, $\widehat{W}_{\cQ} := \Pi_{\Range(\cQ)}(\widehat W) \in \Range(\cQ)$ solves \eqref{eq-qppal-w} with the residual satisfying \eqref{eq-qppal-w-r}. Moreover,
$\cQ \widehat{W}_{\cQ} = \cQ \widehat W$ and  $\inprod{\widehat{W}_{\cQ}}{\cQ \widehat{W}_{\cQ}} = \inprod{\widehat W}{\cQ \widehat W}$.
\end{prop}
\begin{proof} First we note that the results $\cQ \widehat{W}_{\cQ} = \cQ \widehat W$ and $\inprod{\widehat{W}_{\cQ}}{\cQ \widehat{W}_{\cQ}} = \inprod{\widehat W}{\cQ \widehat W}$
follow from the decomposition $\hW = \Pi_{\Range(\cQ)}(\hW) +
\Pi_{\Range(\cQ)^\perp}(\hW)$. Next,
by observing that
\begin{eqnarray*}
   \norm{\cQ R - (\cQ+\sigma \cQ^2)\hW_{\cQ}}
   = \norm{\cQ R - (\cQ+\sigma \cQ^2)\widehat W}
   \le {\lambda_{\max}(\cQ)} \,\norm{R-(\cI+\sigma \cQ)\widehat W }
   \le  \varepsilon,
\end{eqnarray*}
one can easily obtain the desired results.
\end{proof}

\medskip
By Proposition \ref{prop:p1-com}, in order to obtain $\hW_{\cQ}$, we can first apply an iterative method such as the preconditioned conjugate gradient (PCG) method to solve \eqref{eq-p1com-w} to obtain $\hW$ and then perform the projection step.
However, by carefully analysing the steps in {\sc Qsdpnal}-Phase I,
we are {surprised} to observe that
instead of explicitly computing $\hW_{\cQ}$, we can update the iterations in the algorithm  by using only $\cQ \hW_{\cQ}=\cQ\hW$.
Thus, we only need to compute $\cQ\widehat W$ and the {potentially expensive} projection step
to compute $\widehat{W}_\cQ$
{can be avoid completely.}

It is important for us to emphasize the computational advantage of solving the linear system \eqref{eq-p1com-w} over  \eqref{eq-qppal-w}. First, the former only requires
one evaluation of $\cQ(\cdot)$ whereas the latter requires two such evaluations in each
PCG iteration. Second, the coefficient matrix in the former system is typically much
more well-conditioned than the coefficient matrix in the latter system.
More precisely, when $\cQ$ is positive definite, then $\cI+\sig\cQ$ is clearly better
conditioned than $\cQ+\sig\cQ^2$ {by a factor of $\lambda_{\max}(\cQ)/\lambda_{\min}(\cQ)$}. When $\cQ$ is singular, with its smallest
positive eigenvalue denoted as $\lambda_{+}(\cQ)$, then  $\cI+\sig\cQ$ is better
conditioned when $\lambda_{\max}(\cQ) \geq  \lambda_+(\cQ)(1+\sig\lambda_+(\cQ))$.
The previous inequality would obviously hold when  $\lambda_+ \leq (\sqrt{4\sig\lambda_{\max}(\cQ)+1}-1)/(2\sig)$.

In Algorithm {\sc Qsdpal}-Phase II, the subspace constraint $W\in\Range(\cQ)$ also appears
when we solve the semismooth Newton linear system \eqref{eqn-epsk} in
Algorithm SNCG. Specifically, we need to find
$(dW,dy)$ to solve the following linear system
\begin{equation}\label{eq-p2-com-dw}
V\, (dW,dy) + \varrho (0, dy) \approx (\cQ(R_1), R_2), \quad (dW,dy)\in\Range(\cQ)\times\Re^m
\end{equation}
 with the residual satisfying the following condition
 \begin{equation}\label{eq-p2-com-r}
 \norm{V\, (dW,dy) + \varrho (0, dy) - (\cQ(R_1), R_2)} \le \varepsilon,\end{equation}
where
 \begin{equation*}
   V: = \left[ \begin{array}{cc}
        \cQ &   \\
           & 0
   \end{array} \right]+ \sigma \left[ \begin{array}{c}
       \cQ  \\
       -\cA
   \end{array} \right] \cU [\cQ \, -\cA^*],
 \end{equation*}
$\cU$ is a given self-adjoint positive semidefinite linear operator on $\Sn$ and $\varepsilon>0$, $\sigma >0$ and $\varrho  > 0$ are given.
Again, instead of solving \eqref{eq-p2-com-dw} directly, we can solve
a simpler linear system to compute
$\cQ (dW)$ approximately, as shown in the next proposition.
{The price to pay is that we now need to solve nonsymmetric linear system instead of a symmetric one.}

\begin{prop}\label{prop:p2-com}
Let \begin{equation*}
   \widehat V: = \left[ \begin{array}{cc}
        \cI &   \\
           & 0
   \end{array} \right]+ \sigma  \left[ \begin{array}{c}
       \cI  \\
       -\cA
   \end{array} \right] \cU [\cQ \; -\cA^*].
 \end{equation*}
Suppose $(\widehat{dW}, \widehat{dy})$ is an approximate solution to the following system:
\begin{equation}
\label{eq-p2com-dw}
\widehat V\, (dW, dy) + \varrho (0, dy) \approx (R_1, R_2)
\end{equation}
with the residual satisfying
$$
\norm{\widehat V\, (\widehat{dW}, \widehat{dy}) + \varrho (0, \widehat{dy}) - (R_1, R_2)}\le \frac{\varepsilon}{{\max\{\lambda_{\max}(\cQ),1\}}}.
$$
Let $\widehat{dW}_\cQ = \Pi_{\Range(\cQ)}(\widehat{dW}) \in \Range(\cQ)$  Then $(\widehat{dW}_\cQ ,\widehat{dy} )$ solves \eqref{eq-p2-com-dw} with the residual satisfying \eqref{eq-p2-com-r}. Moreover, $\cQ\, \widehat{dW}_\cQ = \cQ\, \widehat{dW}$ and $\inprod{\widehat{dW}_\cQ}{\cQ\, \widehat{dW}_\cQ} = \inprod{\widehat{dW}}{\cQ\, \widehat{dW}}$.
\end{prop}
\begin{proof} The proof that
$\cQ\, \widehat{dW}_\cQ = \cQ\, \widehat{dW}$ and $\inprod{\widehat{dW}_\cQ}{\cQ\, \widehat{dW}_\cQ} = \inprod{\widehat{dW}}{\cQ\, \widehat{dW}}$ is the same as in the previous proposition.
Observe that $V = \Diag(\cQ,\cI)\widehat V$. Then, by using the fact that
\begin{eqnarray*}
 && \hspace{-0.7cm} \norm{V\,(\widehat{dW}_\cQ,\widehat{dy}) + \varrho (0, \widehat{dy}) - (\cQ(R_1), R_2)}
  = \norm{V\,(\widehat{dW},\widehat{dy})+ \varrho (0, \widehat{dy}) - (\cQ(R_1), R_2)} \\[5pt]
  & \le&\norm{\Diag(\cQ,\cI)}_2\,\norm{\widehat V\, (\widehat{dW}, \widehat{dy}) + \varrho (0, \widehat{dy}) - (R_1, R_2)}
\leq
{\max\{\lambda_{\max}(\cQ),1\}}
\frac{\varepsilon}{{\max\{\lambda_{\max}(\cQ),1\}}}  = \varepsilon,
\end{eqnarray*}
we obtain the desired results readily.
\end{proof}

\section{Adaption of QSDPNAL for least squares SDP and
inequality constrained QSDP }
\label{sec:5}

{
Here we discuss how our algorithm \QSDPNAL can be modified and
adapted for solving least squares semidefinite programming
as well as general QSDP problems with additional
unstructured inequality constraints which are not
captured by the polyhedral set $\cK.$

\subsection{The case  for least squares semidefinite programming}\label{sec:LS}
In this subsection, we show that for  least squares semidefinite programming problems, \QSDPNAL can be used in a more efficient  way to 
{avoid}
the difficulty of handling the subspace constraint $W\in\Range(\cQ)$.

Consider the following least squares semidefinite programming problem
 \begin{eqnarray}
  \begin{array}{ll}
    \min \; \Big\{\displaystyle \frac{1}{2} \norm{\cB X - d}^2 + \inprod{C}{X} \, \mid\,
       \cA X   =  b, \;
       X \in \Sn_+\cap \cK  \Big\},
\end{array}
 \label{eq-ls-qsdp}
\end{eqnarray}
where $\cA:\Sn\to\Re^{m}$ and $\cB:\Sn\to\Re^{s}$ are two linear maps, $C\in \Sn$,  $b \in \Re^{m}$ and $d\in\Re^{s}$ are given data,  $\cK$ is a simple nonempty closed convex polyhedral set in $\Sn$.

It is easy to see that \eqref{eq-ls-qsdp}  can be rewritten as follows
 \begin{eqnarray}
  \begin{array}{ll}
    \min \;\Big\{ \displaystyle \frac{1}{2} \norm{u}^2 + \inprod{C}{X}  \,\mid\,
       \cB X - d = u,\; \cA X   =  b, \;
       X \in \Sn_+\cap \cK  \Big\}.
\end{array}
 \label{eq-SMLSp}
\end{eqnarray}
The dual of \eqref{eq-SMLSp} takes the following form
 \begin{eqnarray}
    \max \Big\{ -\delta_{\cK}^*(-Z) -\frac{1}{2}\norm{\xi}^2 + \inprod{d}{\xi} + \inprod{b}{y}\mid
        Z +\cB^*\xi + S +\cA^*y = C,\;\; S\in\Sn_+ \Big\}.
 \label{eq-SMLSd}
\end{eqnarray}

When {\sc Qsdpnal}-Phase I is applied to solve \eqref{eq-SMLSd},
instead of solving \eqref{eq-qppal-w}, the linear system corresponding to the quadratic term is given by
\begin{equation}\label{LS-p1}
(\cI + \sigma\cB\cB^*)\xi \approx R,\end{equation}
where $R\in\Re^s$ and $\sigma >0$ are given data.
Meanwhile, in {\sc Qsdpnal}-Phase II for solving  problem \eqref{eq-SMLSd}, the linear system in the SNCG method
 is given by
\begin{equation}
\label{LS-p2}
\left(\left[
          \begin{array}{cc}
            \cI &   \\
             & 0
          \end{array}
        \right] + \sigma\left[
                          \begin{array}{c}
                            \cB \\
                            \cA \\
                          \end{array}
                        \right]\cU\left[
                              \begin{array}{cc}
                                \cB^* & \cA^*\\
                              \end{array}
                            \right]
\right)\left[
         \begin{array}{c}
           d\xi \\
           dy \\
         \end{array}
       \right] \approx \left[
         \begin{array}{c}
            R_1\\
            R_2  \\
         \end{array}
       \right],
\end{equation}
where $R_1\in\Re^s$ and $R_2\in\Re^{m}$ are given data, $\cU$ is a given self-adjoint positive semidefinite linear operator on $\Sn$.
{It is clear that just like \eqref{eq-p1com-w},
one can solve \eqref{LS-p1} efficiently via the PCG method.
For \eqref{LS-p2}, one can also solve it by the PCG method, which is more appealing
compared to using a nonsymmetric  iterative solver such as the preconditioned
BiCGSTAB to solve the nonsymmetric linear system \eqref{eq-p2com-dw}.
}

\begin{remark}
  \label{rmk:partial-ppa}
  When the polyhedral constraint $X\in\cK$ in \eqref{eq-ls-qsdp} is absent, i.e., the polyhedral convex set $\cK = \Sn$, Jiang, Sun and Toh in \cite{jiang2014partial} {have
  proposed a} partial proximal point algorithm for solving  the least squares semidefinite programming problem \eqref{eq-ls-qsdp}. {Here our Algorithm {\sc Qsdpnal} is built to  solve the much more  general class of convex composite QSDP problems.}
\end{remark}

\subsection{Extension to QSDP problems with inequality constraints} \label{sec:5.2}

Consider the following general QSDP problem:
 \begin{equation}
    \min\, \Big\{ \displaystyle\frac{1}{2} \inprod{X}{\cQ X} + \inprod{C}{X} \mid
       \cA_E X   =  b_E,
       \;\cA_I X    \le b_I, \;
       X \in \Sn_+\cap \cK  \Big\},
 \label{eq-qsdp}
\end{equation}
where $\cA_E:\Sn\to\Re^{m_E}$ and $\cA_I:\Sn\to\Re^{m_I}$ are two linear maps.
By adding a slack variable $x$, we can equivalently rewrite \eqref{eq-qsdp} into the
following standard form:
 \begin{eqnarray}
  \begin{array}{ll}
    \min & \displaystyle \frac{1}{2} \inprod{X}{\cQ X} + \inprod{C}{X}  \\[5pt]
   \mbox{s.t.}
       &\cA_E X   =  b_E,
       \quad \cA_I X +\cD x = b_I, \quad
       X \in \Sn_+\cap \cK,  \quad  \cD x \ge 0,
\end{array}
 \label{eq-qsdp-standard}
\end{eqnarray}
where $\cD:\Re^{m_I}\to\Re^{m_I}$ is a
positive definite diagonal matrix  which is introduced for the purpose of
scaling the variable $x$.
The
dual of \eqref{eq-qsdp-standard} is given by
\begin{equation}
  \begin{array}{rllll}
    \max & \displaystyle  -\delta^*_{\cK}(-Z)  -\frac{1}{2}\inprod{W}{\cQ W}  +  \inprod{b_E}{y_E} + \inprod{b_I}{y_I}     \\[5pt]
   \mbox{s.t.} & Z   - \cQ W + S + \cA_E^* y_E + \cA_I^*y_I  = C,  \\[5pt]
   & \cD^*(s + y_I) = 0, \quad S\in\Sn_+,\quad s\ge 0, \quad W\in\Range(\cQ).
   \end{array}
   \label{eq-d-qsdp-standard}
\end{equation}
We can express \eqref{eq-d-qsdp-standard} in a form which is similar  to ({\bf D}) as
follows:
\begin{eqnarray}
  \begin{array}{rllll}
    \max & \displaystyle  -\delta^*_{\cK}(-Z)  -\frac{1}{2}\inprod{W}{\cQ W}  +  \inprod{b_E}{y_E} + \inprod{b_I}{y_I}
\\[8pt]
   \mbox{s.t.} &
\left( \begin{array}{c}
 \cI \\[3pt] 0
\end{array} \right) Z
- \left( \begin{array}{c}
 \cQ \\[3pt] 0
\end{array} \right) W +
\left( \begin{array}{cc}
  \cI & 0  \\[3pt] 0 & \cD^*
\end{array} \right)
\left( \begin{array}{c}
 S \\[3pt] s
\end{array} \right) +
\left( \begin{array}{cc}
 \cA^*_E & \cA_I^* \\[3pt] 0 & \cD^*
\end{array} \right)
 \left( \begin{array}{c}
  y_E\\[3pt] y_I
\end{array} \right)
 = \left( \begin{array}{c}
 C \\[3pt] 0
\end{array} \right),
\\[12pt]
 & (S,s)\in\Sn_+\times \Re^{m_I}_+, \quad W\in\Range(\cQ).
   \end{array}
   \label{eq-d-qsdp-D}
\end{eqnarray}
We can readily extend \QSDPNAL to solve the above more general form of \eqref{eq-d-qsdp-D},
and our implementation of \QSDPNAL indeed can be used to solve \eqref{eq-d-qsdp-D}.

}

\section{Computational experiments}\label{sec:comp-example}
\label{sec:5.2}

In this section, we evaluate the performance our algorithm \QSDPNAL for solving large-scale QSDP problems \eqref{eq-qsdp}.  Since \QSDPNAL contains two phases,
 we also report the numerical results obtained by running {\sc Qsdpnal}-Phase I (a first-order algorithm) alone for the purpose of demonstrating the power and importance of our two-phase framework for solving difficult QSDP problems.
In the numerical experiments, we measure the accuracy of an approximate optimal solution $(X,Z,W,S,y_E,y_I)$ for QSDP \eqref{eq-qsdp} and its dual by using the following relative KKT residual:
\begin{eqnarray*}
  \eta_{\textup{qsdp}} = \max\{\eta_P, \eta_D, \eta_Z, \eta_{S_1}, \eta_{S_2}, \eta_{I_1}, \eta_{I_2},\eta_{I_3},\eta_W\},
  \label{stop:sqsdp}
\end{eqnarray*}
where
{\small
\begin{eqnarray*}
&& \eta_P = \frac{\norm{b_E - \cA_E X}}{1+\norm{b_E}},\quad \eta_D = \frac{\norm{Z - \cQ W  + S + \cA_E^* y_E + \cA_I^* y_I - C}}{1 + \norm{C}}, \quad
\eta_{Z} = \frac{\norm{X - \Pi_{\cK}(X-Z)}}{1+\norm{X}+\norm{Z}}, \\[5pt]
&& \eta_{S_1} = \frac{|\inprod{S}{X}|}{1+\norm{S}+\norm{X}}, \quad \eta_{S_2} = \frac{\norm{X - \Pi_{\Sn_+}(X)}}{1+\norm{X}},\quad \eta_{I_1} = \frac{\norm{\min(b_I - \cA_I X,0)}}{1+\norm{b_I}},\quad \eta_{I_2} = \frac{\norm{\max(y_I,0)}}{1+\norm{y_I}},\\[5pt]
&& \eta_{I_3} = \frac{|\inprod{b_I-\cA_I X}{y_I}|}{1+\norm{y_I}+\norm{b_I - \cA_I X}},\quad
\eta_{W} = \frac{\norm{\cQ W - \cQ X}}{1+\norm{\cQ}}.
\end{eqnarray*}}
Additionally, we also compute the relative duality gap defined by
\[\eta_{\textup{gap}} = \frac{\textup{obj}_P-\textup{obj}_D}{1+|\textup{obj}_P|+|\textup{obj}_D|},\]
where $\textup{obj}_P := \frac{1}{2}\inprod{X}{\cQ X} + \inprod{C}{X}$ and
$\textup{obj}_D := -\delta_{\cK}^*(-Z) - \frac{1}{2}\inprod{W}{\cQ W} + \inprod{b_E}{y_E}
+\inprod{b_I}{y_I}$.  We terminate both \QSDPNAL and {\sc Qsdpnal}-Phase I when $\eta_{{\rm qsdp}} < 10^{-6}$ with the maximum number of iterations set at 50,000.

In our implementation of {\sc Qsdpnal}, we always run {\sc Qsdpnal}-Phase I first to
{generate} a reasonably good starting point to warm start our Phase II algorithm. We terminate the Phase I algorithm and switch to the Phase II algorithm if a solution with a moderate accuracy (say a solution with $\eta_{\rm qsdp} < 10^{-4}$) is obtained or if the Phase I algorithm reaches the maximum number of iterations (say 1000 iterations).
 If the {underlying} problems contain inequality or polyhedral constraints,
 we further employ a restarting strategy similar to the one in \cite{YangST2015}, i.e., when  the progress of {\sc Qsdpnal}-Phase II  is not satisfactory,
we will restart {the whole {\sc Qsdpnal} algorithm}  by using the most recently computed $(Z,W,S,y,X,\sigma)$
as the initial point.
{In addition, we also adopt a dynamic tuning strategy to adjust the penalty parameter $\sigma$
appropriately based on the progress of the primal and dual feasibilities of the computed iterates.}

All our computational results  are obtained from a workstation running on 64-bit Windows
 Operating System having
16 cores  with 32 Intel Xeon E5-2650 processors at 2.60GHz and 64 GB memory. We have implemented \QSDPNAL in {\sc Matlab} version 7.13.

\subsection{Evaluation of \QSDPNAL on the nearest correlation matrix problems}

Our first test example is the problem  of finding the nearest correlation matrix (NCM) to a given
matrix $G \in \Sn$:
\begin{eqnarray}
    \begin{array}{ll}
    \min   \Big\{ \displaystyle\frac{1}{2}\norm{H\circ(X-G)}^2_F   \mid
   {\rm diag}(X) \;=\; e, \;
       X \in \Sn_+\cap \cK  \Big\},
\end{array} \label{eq-egHNCM-F}
\end{eqnarray}
where $H\in \Sn$ is a nonnegative weight matrix, $e\in\Re^n$ is the vector of all ones, and $\cK =\{W\in\Sn\;|\; L\leq W\leq U\}$ with $L,U\in \Sn$ being given matrices.

In our numerical experiments, we first take a matrix $\widehat{G}$, which is a correlation matrix generated from gene expression data from \cite{li2010inexact}.
For testing purpose, we then perturb $\widehat{G}$ to
\begin{eqnarray*}
  G := (1 - \alpha)\widehat{G} + \alpha E,
\end{eqnarray*}
where $\alpha \in (0,1)$ is a given parameter and $E$ is a randomly generated symmetric matrix with entries uniformly distributed in $[-1,1]$ except for its diagonal elements which are
all set to $1$.   The weight matrix $H$ is generated
from a weight matrix $H_0$ used by a hedge fund company.
The matrix $H_0$ is a $93 \times 93$ symmetric matrix with all positive entries. It has about $24\%$ of the entries equal to
$10^{-5}$ and the rest are distributed in the interval $[2, 1.28\times 10^3].$
 The {\sc Matlab} code for generating the matrix $H$ is given by
\begin{verbatim}
      tmp = kron(ones(110,110),H0); H = tmp(1:n,1:n); H = (H'+H)/2.
\end{verbatim}
The reason for using such a weight matrix is because the resulting problems
generated are more challenging to solve as opposed to a randomly generated
weight matrix.
{We also test four more instances, namely {\tt PDidx2000}, {\tt PDidx3000}, {\tt PDidx5000} and
{\tt PDidx10000}, where the raw correlation matrix
$\widehat G$ is generated from the probability of default (PD) data obtained from the RMI Credit Research Initiative\footnote{\url{http://www.rmicri.org/cms/cvi/overview/.}} at the National University of Singapore.}
We consider two choices of $\cK$,
i.e., case (i): $\cK = \Sn$ and case (ii): $\cK = \{X\in \Sn \mid\, X_{ij} \ge -0.5, \;\forall\; i,j=1,\ldots,n\}$.

\begin{table}
\centering
\begin{footnotesize}
\caption{{\small The performance of \QSDPNAL and {\sc Qsdpnal}-Phase I on 
H-weighted NCM problems (dual of \eqref{eq-egHNCM-F}) (accuracy $= 10^{-6}$). In the table, ``a'' stands for \QSDPNAL and ``b'' stands for {\sc Qsdpnal}-Phase I, respectively. The computation time is in the format of ``hours:minutes:seconds''.}}
\label{table:ncm_F}
\begin{tabular}{| ccc | c |c|  c | c| c|}
\hline
\mc{8}{|c|}{}\\[-1pt]
\mc{8}{|c|}{ $\cK = \Sn$}\\[2pt]
\hline
\mc{3}{|c|}{} & \mc{1}{c|}{iter.a} &\mc{1}{c|}{iter.b} &\mc{1}{c|}{$\eta_{\textup{qsdp}}$}&\mc{1}{c|}
{$\eta_\textup{gap}$}&\mc{1}{c|}{time}\\[2pt] \hline
\mc{1}{|c}{problem} &\mc{1}{c}{$n$} &\mc{1}{c|}{$\alpha$}&\mc{1}{c|}{it (subs) $|$ itSCB}&\mc{1}{c|}{}&\mc{1}{c|}{a$|$b}
&\mc{1}{c|}{a$|$b}&\mc{1}{c|}{a$|$b}\\[2pt]
\hline
\mc{1}{|c}{Lymph}
	 &587 	 &0.10 	 &12  (40)  $|$  52 	 &251 	 &   9.1-7 $|$    9.1-7	 &   8.2-7 $|$    -3.9-7 	 &13 $|$ 23\\[2pt]
\mc{1}{|c}{Lymph}
	 &587 	 &0.05 	 &11  (32)  $|$  38 	 &205 	 &   9.5-7 $|$    9.9-7	 &   7.5-7 $|$    -4.1-7 	 &09 $|$ 19\\[2pt]
 \hline
\mc{1}{|c}{ER}
	 &692 	 &0.10 	 &12  (41)  $|$  54 	 &250 	 &   9.8-7 $|$    9.9-7	 &   5.4-7 $|$    -4.8-7 	 &17 $|$ 33\\[2pt]
\mc{1}{|c}{ER}
	 &692 	 &0.05 	 &12  (38)  $|$  43 	 &218 	 &   7.3-7 $|$    9.7-7	 &   2.5-7 $|$    -4.4-7 	 &14 $|$ 28\\[2pt]
 \hline
\mc{1}{|c}{Arabidopsis}
	 &834 	 &0.10 	 &12  (42)  $|$  56 	 &285 	 &   8.5-7 $|$    9.9-7	 &   2.8-7 $|$    -5.3-7 	 &27 $|$ 57\\[2pt]
\mc{1}{|c}{Arabidopsis}
	 &834 	 &0.05 	 &12  (41)  $|$  44 	 &230 	 &   8.0-7 $|$    9.5-7	 &   -6.8-8 $|$    -4.5-7 	 &24 $|$ 46\\[2pt]
 \hline
\mc{1}{|c}{Leukemia}
	 &1255 	 &0.10 	 &12  (41)  $|$  62 	 &340 	 &   8.4-7 $|$    9.9-7	 &   3.1-7 $|$    -5.4-7 	 &1:08 $|$ 2:48\\[2pt]
\mc{1}{|c}{Leukemia}
	 &1255 	 &0.05 	 &12  (38)  $|$  49 	 &248 	 &   7.6-7 $|$    8.7-7	 &   -1.3-7 $|$    -4.5-7 	 &58 $|$ 2:06\\[2pt]
 \hline
\mc{1}{|c}{hereditarybc}
	 &1869 	 &0.10 	 &13  (47)  $|$  76 	 &393 	 &   6.4-7 $|$    9.9-7	 &   -2.2-7 $|$    -9.8-7 	 &3:01 $|$ 7:10\\[2pt]
\mc{1}{|c}{hereditarybc}
	 &1869 	 &0.05 	 &13  (45)  $|$  60 	 &311 	 &   8.6-7 $|$    9.9-7	 &   -4.7-7 $|$    {\green{ -1.0-6}} 	 &2:39 $|$ 5:44\\[2pt]
 \hline
\mc{1}{|c}{PDidx2000}
	 &2000 	 &0.10 	 &13  (51)  $|$  131 	 &590 	 &   9.5-7 $|$    9.9-7	 &   2.4-7 $|$    -8.5-7 	 &5:04 $|$ 11:43\\[2pt]
\mc{1}{|c}{PDidx2000}
	 &2000 	 &0.05 	 &14  (58)  $|$  139 	 &626 	 &   7.5-7 $|$    9.9-7	 &   -5.6-8 $|$    -9.5-7 	 &5:52 $|$ 12:41\\[2pt]
 \hline
\mc{1}{|c}{PDidx3000}
	 &3000 	 &0.10 	 &14  (55)  $|$  145 	 &1201 	 &   8.1-7 $|$    9.9-7	 &   -2.8-7 $|$    {\green{ 2.1-6}} 	 &14:59 $|$   1:15:01\\[2pt]
\mc{1}{|c}{PDidx3000}
	 &3000 	 &0.05 	 &14  (58)  $|$  136 	 &1263 	 &   6.8-7 $|$    9.7-7	 &   -2.6-7 $|$    {\green{ 2.0-6}} 	 &14:50 $|$   1:19:27\\[2pt]
 \hline
\mc{1}{|c}{PDidx5000}
	 &5000 	 &0.10 	 &15  (63)  $|$  189 	 &1031 	 &   8.0-7 $|$    9.9-7	 &   -1.9-7 $|$    {\green{ 1.8-6}} 	 &  1:17:47 $|$   4:17:10\\[2pt]
\mc{1}{|c}{PDidx5000}
	 &5000 	 &0.05 	 &14  (59)  $|$  164 	 &1699 	 &   9.2-7 $|$    9.9-7	 &   -3.3-7 $|$    -1.3-7 	 &  1:11:46 $|$   6:18:29\\[2pt]
 \hline
\mc{1}{|c}{PDidx10000}
	 &10000 	 &0.10 	 &16  (71)  $|$  200 	 &2572 	 &   7.1-7 $|$    9.9-7	 &   1.6-7 $|$    -1.5-7 	 &  9:57:18 $|$   60:07:08\\[2pt]
\mc{1}{|c}{PDidx10000}
	 &10000 	 &0.05 	 &16  (73)  $|$  200 	 &2532 	 &   9.5-7 $|$    9.9-7	 &   4.7-8 $|$    1.4-7 	 &  10:34:31 $|$   59:34:13\\[2pt]
 \hline

\mc{8}{|c|}{}\\[-1pt]
\mc{8}{|c|}{$\cK = \{X\in \Sn \mid\, X_{ij}\ge -0.5\;\forall\; i,j=1,\ldots,n\}$}\\[2pt]
\hline
\mc{1}{|c}{Lymph}
	 &587 	 &0.10 	 & 5  (14)  $|$  129 	 &244 	 &   9.8-7 $|$    9.9-7	 &   -1.0-7 $|$    -4.4-7 	 &18 $|$ 30\\[2pt]
\mc{1}{|c}{Lymph}
	 &587 	 &0.05 	 & 5  (12)  $|$  120 	 &257 	 &   9.9-7 $|$    9.9-7	 &   -3.4-7 $|$    -4.2-7 	 &15 $|$ 28\\[2pt]
 \hline
\mc{1}{|c}{ER}
	 &692 	 &0.10 	 & 5  (14)  $|$  126 	 &266 	 &   9.9-7 $|$    9.9-7	 &   -1.5-7 $|$    -5.1-7 	 &22 $|$ 40\\[2pt]
\mc{1}{|c}{ER}
	 &692 	 &0.05 	 & 5  (14)  $|$  117 	 &217 	 &   8.4-7 $|$    9.9-7	 &   -2.7-7 $|$    -4.4-7 	 &21 $|$ 32\\[2pt]
 \hline
\mc{1}{|c}{Arabidopsis}
	 &834 	 &0.10 	 & 6  (16)  $|$  240 	 &472 	 &   9.9-7 $|$    9.9-7	 &   -5.4-7 $|$    -6.0-7 	 &1:03 $|$ 1:56\\[2pt]
\mc{1}{|c}{Arabidopsis}
	 &834 	 &0.05 	 & 6  (15)  $|$  240 	 &442 	 &   8.5-7 $|$    9.9-7	 &   -4.4-7 $|$    -5.6-7 	 &1:02 $|$ 1:46\\[2pt]
 \hline
\mc{1}{|c}{Leukemia}
	 &1255 	 &0.10 	 & 7  (22)  $|$  188 	 &333 	 &   9.9-7 $|$    9.9-7	 &   -4.4-7 $|$    -5.5-7 	 &2:10 $|$ 3:06\\[2pt]
\mc{1}{|c}{Leukemia}
	 &1255 	 &0.05 	 & 7  (19)  $|$  159 	 &253 	 &   9.9-7 $|$    9.9-7	 &   -5.4-7 $|$    -5.3-7 	 &1:46 $|$ 2:18\\[2pt]
 \hline
\mc{1}{|c}{hereditarybc}
	 &1869 	 &0.10 	 & 8  (22)  $|$  397 	 &577 	 &   9.3-7 $|$    9.9-7	 &   -8.0-7 $|$    -8.9-7 	 &10:28 $|$ 12:59\\[2pt]
\mc{1}{|c}{hereditarybc}
	 &1869 	 &0.05 	 & 8  (22)  $|$  361 	 &472 	 &   9.6-7 $|$    9.9-7	 &   -8.1-7 $|$    -8.6-7 	 &9:39 $|$ 10:04\\[2pt]
 \hline
\mc{1}{|c}{PDidx2000}
	 &2000 	 &0.10 	 &20  (52)  $|$  672 	 &716 	 &   9.9-7 $|$    9.9-7	 &   -6.8-7 $|$    -7.9-7 	 &21:32 $|$ 17:42\\[2pt]
\mc{1}{|c}{PDidx2000}
	 &2000 	 &0.05 	 &22  (60)  $|$  756 	 &1333 	 &   9.6-7 $|$    5.8-7	 &   -6.3-7 $|$    -4.0-7 	 &25:20 $|$ 39:34\\[2pt]
 \hline
\mc{1}{|c}{PDidx3000}
	 &3000 	 &0.10 	 &34  (101)  $|$  659 	 &1647 	 &   9.9-7 $|$    9.9-7	 &   -7.0-7 $|$    -9.4-7 	 &  1:14:15 $|$   1:53:13\\[2pt]
\mc{1}{|c}{PDidx3000}
	 &3000 	 &0.05 	 &41  (117)  $|$  728 	 &1538 	 &   9.9-7 $|$    9.9-7	 &   -6.2-7 $|$    {\green{ -1.2-6}} 	 &  1:21:13 $|$   1:50:47\\[2pt]
 \hline
\mc{1}{|c}{PDidx5000}
	 &5000 	 &0.10 	 &29  (79)  $|$  829 	 &1484 	 &   9.3-7 $|$    8.4-7	 &   -5.5-7 $|$    6.4-7 	 &  5:00:35 $|$   7:25:19\\[2pt]
\mc{1}{|c}{PDidx5000}
	 &5000 	 &0.05 	 &33  (107)  $|$  1081 	 &1722 	 &   9.9-7 $|$    9.9-7	 &   -6.4-7 $|$    -1.5-7 	 &  6:30:35 $|$   7:16:08\\[2pt]
 \hline
\mc{1}{|c}{PDidx10000}
	 &10000 	 &0.10 	 &42 (136) $|$  1289 	 &2190 	 &   9.9-7 $|$    9.9-7	 &   -7.1-7 $|$    2.6-7 	 &  58:44:49 $|$   64:17:14\\[2pt]
\mc{1}{|c}{PDidx10000}
	 &10000 	 &0.05 	 &40 (122)  $|$  1519 	 &3320 	 &   9.9-7 $|$    4.5-7	 &   -6.6-7 $|$    -1.7-8 	 &  65:13:19 $|$   94:53:10\\[2pt]
 \hline

\end{tabular}
\end{footnotesize}
\end{table}

In Table \ref{table:ncm_F}, we report the numerical results obtained by \QSDPNAL and {\sc Qsdpnal}-Phase I in solving various instances of the H-weighted NCM problem \eqref{eq-egHNCM-F}. In the table, ``it (subs)'' denotes the number of outer iterations with subs in the parenthesis indicating the number of inner iterations   of {\sc Qsdpnal}-Phase II and ``itSCB'' stands for the total number of iterations used in {\sc Qsdpnal}-Phase I.
We can see from Table \ref{table:ncm_F}
that \QSDPNAL is more efficient than the purely first-order algorithm {\sc Qsdpnal}-Phase I.
In particular, for the instance {\tt PDidx10000} where the matrix dimension $n=10,000$, we are able to solve
the problem in about 11 hours while the purely first-order method {\sc Qsdpnal}-Phase I needs about 60 hours.

\subsection{Evaluation of \QSDPNAL on instances generated from BIQ problems}
 Based on the SDP relaxation of a binary integer quadratic (BIQ) problem considered in \cite{SunTY3c}, we construct our second QSDP test example as following:
 \begin{equation*}
  \begin{array}{rl}
 \mbox{(QSDP-BIQ)}\;\;   \min & \displaystyle  \frac{1}{2}\inprod{X}{\cQ X} + \frac{1}{2} \inprod{Q}{Y} + \inprod{c}{x}  \\[5pt]
   \mbox{s.t.}
       & \textup{diag}(Y) - x = 0, \quad \alpha = 1, \quad  X = \left(
               \begin{array}{cc}
                 Y & x \\
                 x^T & \alpha \\
               \end{array}
             \right)
       \in \Sn_+, \quad X\in \cK, \\[5pt]
       & -Y_{ij}+x_i\ge 0, \, -Y_{ij}+x_j\ge0,\, Y_{ij}-x_i-x_j\ge -1,\, \forall \, i<j, \, j=2,\ldots,n-1,
       \end{array}
  \label{qsdp-biq}
\end{equation*}
where the convex set $\cK = \{X \in \cS^{n} \mid X \geq 0 \}$.
Here $\cQ:\Sn\to\Sn$ is a self-adjoint positive semidefinite linear operator defined by
\begin{equation}\label{Qmap}
\cQ(X) = \frac{1}{2}(AXB + BXA)
\end{equation}
with $A,B\in \Sn_+$ being matrices truncated from two different large correlation matrices (generated from Russell 1000 and Russell 2000 index, respectively) fetched from Yahoo finance by {\sc Matlab}. In our numerical experiments, the test data for $Q$ and $c$ are taken from Biq Mac Library maintained by Wiegele, which is available at \url{http://biqmac.uni-klu.ac.at/biqmaclib.html}.

Table \ref{table:BIQI} reports the numerical results for \QSDPNAL and {\sc Qsdpnal}-Phase I in solving some large scale QSDP-BIQ problems. Note that from the numerical experiments conducted in  \cite{chen2015efficient}, one can clearly conclude that {\sc Qsdpnal}-Phase I (a variant of SCB-isPADMM)
is the most efficient first-order algorithm for solving QSDP-BIQ problems with a large number of inequality constraints. Even so, it can be observed from Table \ref{table:BIQI} that \QSDPNAL is still faster than {\sc Qsdpnal}-Phase I on most of the problems tested.

\begin{footnotesize}
\begin{longtable}{| c c | c |  c | c| c|c|}
\caption{The performance of \QSDPNAL and {\sc Qsdpnal}-Phase I on QSDP-BIQ problems
(accuracy $= 10^{-6}$). In the table, ``a'' stands for \QSDPNAL and ``b'' stands for {\sc Qsdpnal}-Phase I, respectively.  The computation time is in the format of ``hours:minutes:seconds''.}\label{table:BIQI}
\\
\hline
 \mc{2}{|c|}{} &\mc{1}{c|}{} &\mc{1}{c|}{}&\mc{1}{c|}{}&\mc{1}{c|}{}&\mc{1}{c|}{}\\[-5pt]
\mc{2}{|c|}{} & \mc{1}{c|}{iter.a} &\mc{1}{c|}{iter.b} &\mc{1}{c|}{$\eta_{\textup{qsdp}}$}
&\mc{1}{c|}{$\eta_\textup{gap}$}&\mc{1}{c|}{time}\\[2pt] \hline
\mc{1}{|@{}c@{}}{$d$} &\mc{1}{@{}c@{}|}{$m_E;m_I$ $|$ $n$} &\mc{1}{c|}{it (subs)$|$itSCB}
&\mc{1}{c|}{}&\mc{1}{c|}{a$|$b}&\mc{1}{c|}{a$|$b}
&\mc{1}{c|}{a$|$b}\\ \hline
\endhead

be200.3.1 
&201  $;$  59700  $|$  201 	 &66  (135) $|$ 3894    &4701 	 &   7.8-7 $|$    9.8-7	 &   -3.5-7 $|$    -7.2-7 	 &3:37 $|$ 3:57\\[2pt] 
\hline 
be200.3.2 
&201  $;$  59700  $|$  201 	 &37  (74) $|$ 2969    &13202 	 &   9.7-7 $|$    9.9-7	 &   -2.1-7 $|$    -6.7-8 	 &2:42 $|$ 12:20\\[2pt] 
\hline 
be200.3.3 
&201  $;$  59700  $|$  201 	 &51  (107) $|$ 5220    &10375 	 &   8.1-7 $|$    9.9-7	 &   -1.1-7 $|$    -6.5-7 	 &5:00 $|$ 8:52\\[2pt] 
\hline 
be200.3.4 
&201  $;$  59700  $|$  201 	 &36  (72) $|$ 3484    &4966 	 &   9.8-7 $|$    9.9-7	 &   -1.6-7 $|$    -4.1-7 	 &3:15 $|$ 4:14\\[2pt] 
\hline 
be200.3.5 
&201  $;$  59700  $|$  201 	 &22  (44) $|$ 2046    &3976 	 &   9.8-7 $|$    9.9-7	 &   -5.9-8 $|$    -3.0-7 	 &1:53 $|$ 3:28\\[2pt] 
\hline 
be250.1 
&251  $;$  93375  $|$  251 	 &98  (196) $|$ 6931    &12220 	 &   9.9-7 $|$    9.9-7	 &   3.2-7 $|$    3.5-8 	 &8:11 $|$ 14:07\\[2pt] 
\hline 
be250.2 
&251  $;$  93375  $|$  251 	 &81  (169) $|$ 6967    &16421 	 &   9.3-7 $|$    9.9-7	 &   3.2-7 $|$    -5.7-7 	 &8:35 $|$ 20:01\\[2pt] 
\hline 
be250.3 
&251  $;$  93375  $|$  251 	 &123  (250) $|$ 7453    &9231 	 &   9.3-7 $|$    9.8-7	 &   -1.7-7 $|$    -5.1-7 	 &9:27 $|$ 10:25\\[2pt] 
\hline 
be250.4 
&251  $;$  93375  $|$  251 	 &36  (72) $|$ 3583    &4542 	 &   9.9-7 $|$    9.9-7	 &   5.2-8 $|$    -2.1-7 	 &4:31 $|$ 5:06\\[2pt] 
\hline 
be250.5 
&251  $;$  93375  $|$  251 	 &99  (198) $|$ 5004    &12956 	 &   8.3-7 $|$    9.9-7	 &   1.8-7 $|$    -1.8-7 	 &6:38 $|$ 15:52\\[2pt] 
\hline 
bqp500-1 
&501  $;$  374250  $|$  501 	 &62  (131) $|$ 5220    &11890 	 &   9.9-7 $|$    9.9-7	 &   -7.1-7 $|$    -8.2-8 	 &37:56 $|$   1:23:58\\[2pt] 
\hline 
bqp500-2 
&501  $;$  374250  $|$  501 	 &41  (84) $|$ 3610    &8159 	 &   5.5-7 $|$    9.9-7	 &   -3.8-7 $|$    -8.7-8 	 &24:01 $|$ 55:14\\[2pt] 
\hline 
bqp500-3 
&501  $;$  374250  $|$  501 	 &89  (200) $|$ 5877    &6402 	 &   9.9-7 $|$    8.6-7	 &   5.4-7 $|$    -1.9-7 	 &40:29 $|$ 41:51\\[2pt] 
\hline 
bqp500-4 
&501  $;$  374250  $|$  501 	 &95  (256) $|$ 7480    &11393 	 &   6.3-7 $|$    9.9-7	 &   -1.5-7 $|$    -1.1-7 	 &56:12 $|$   1:17:56\\[2pt] 
\hline 
bqp500-5 
&501  $;$  374250  $|$  501 	 &107  (247) $|$ 6976    &8823 	 &   5.1-7 $|$    9.9-7	 &   6.2-7 $|$    -1.0-7 	 &52:24 $|$ 59:11\\[2pt] 
\hline 
bqp500-6 
&501  $;$  374250  $|$  501 	 &159  (412) $|$ 10461    &9587 	 &   8.3-7 $|$    9.9-7	 &   -6.2-7 $|$    -1.3-7 	 &  1:18:11 $|$   1:04:41\\[2pt] 
\hline 
bqp500-7 
&501  $;$  374250  $|$  501 	 &92  (223) $|$ 8585    &9066 	 &   8.1-7 $|$    9.9-7	 &   4.7-8 $|$    -1.1-7 	 &  1:00:52 $|$   1:00:35\\[2pt] 
\hline 
bqp500-8 
&501  $;$  374250  $|$  501 	 &68  (140) $|$ 5828    &7604 	 &   6.7-7 $|$    9.9-7	 &   -4.7-8 $|$    -1.1-7 	 &40:56 $|$ 51:58\\[2pt] 
\hline 
bqp500-9 
&501  $;$  374250  $|$  501 	 &50  (108) $|$ 4704    &11613 	 &   9.5-7 $|$    9.9-7	 &   -3.7-7 $|$    -9.8-8 	 &34:05 $|$   1:21:17\\[2pt] 
\hline 
bqp500-10 
&501  $;$  374250  $|$  501 	 &71  (163) $|$ 6462    &8474 	 &   8.7-7 $|$    9.9-7	 &   -6.2-7 $|$    -8.7-8 	 &48:07 $|$ 57:33\\[2pt] 
\hline 
gka1e 
&201  $;$  59700  $|$  201 	 &74  (163) $|$ 5352    &9071 	 &   9.2-7 $|$    9.9-7	 &   -3.0-7 $|$    -2.9-7 	 &7:59 $|$ 9:35\\[2pt] 
\hline 
gka2e 
&201  $;$  59700  $|$  201 	 &49  (98) $|$ 4008    &6659 	 &   9.2-7 $|$    9.9-7	 &   5.0-8 $|$    -1.7-7 	 &4:17 $|$ 6:29\\[2pt] 
\hline 
gka3e 
&201  $;$  59700  $|$  201 	 &35  (71) $|$ 2731    &4103 	 &   8.3-7 $|$    9.7-7	 &   2.3-7 $|$    -2.2-8 	 &2:59 $|$ 4:14\\[2pt] 
\hline 
gka4e 
&201  $;$  59700  $|$  201 	 &34  (68) $|$ 2999    &3430 	 &   9.9-7 $|$    9.9-7	 &   -1.7-7 $|$    -4.6-7 	 &3:20 $|$ 3:21\\[2pt] 
\hline 
gka5e 
&201  $;$  59700  $|$  201 	 &43  (90) $|$ 3367    &2712 	 &   9.9-7 $|$    9.9-7	 &   -4.9-8 $|$    -6.5-8 	 &3:54 $|$ 2:47\\[2pt] 
\hline

\end{longtable}
\end{footnotesize}

\subsection{Evaluation of \QSDPNAL on instances generated from QAP problems}

Next we test the following QSDP problem
motivated from the SDP relaxation of a quadratic assignment problem (QAP) considered
in \cite{povh2009copositive}. The SDP relaxation we used is
adopted from \cite{YangST2015} but we add a convex quadratic term in
the objective to modify it into a QSDP problem. Specifically, given the data matrices
$A_1,A_2\in\cS^l$ of a QAP problem, the  problem we test is
given by:
\begin{eqnarray*}
  \begin{array}{rl}
 \mbox{(QSDP-QAP)}\;\;   \min & \displaystyle\frac{1}{2}\inprod{X}{\cQ X}+\inprod{A_2 \otimes A_1}{X} \\[5pt]
  {\rm s.t.} &  \sum_{i=1}^l X^{ii} = I, \
  \inprod{I}{X^{ij}} = \delta_{ij} \quad \forall \, 1\leq i \leq j\leq l, \\[5pt]
  &
  \inprod{E}{X^{ij}} = 1\quad \forall\, 1\leq i \leq j\leq l,
  \quad X \in \cS^{n}_+,\; X \in \cK,
  \end{array}
  \label{p2:qsdp-qap}
\end{eqnarray*}
where $n=l^2$, and $X^{ij}\in \Re^{l\times l}$ denotes the $(i,j)$-th block of
$X$ when it is partitioned uniformly into an $l\times l$ block matrix with each block
having dimension $l\times l$.
The convex set $\cK = \{X \in \cS^{n} \mid X \geq 0 \}$,
$E$ is the matrix of ones, and $\delta_{ij} = 1$ if $i=j$,
and $0$ otherwise.
Note that here we use the same self-adjoint positive semidefinite linear operator $\cQ:\Sn\to\Sn$ constructed in \eqref{Qmap}.
 In our numerical experiments, the test instances $(A_1,A_2)$ are taken
from the QAP Library \cite{burkard1997qaplib}.

In Table \ref{table:qsdpnal}, we present the detail
numerical results for \QSDPNAL and {\sc Qsdpnal}-Phase I in solving some large scale QSDP-QAP problems.
It is interesting to note that \QSDPNAL can solve all the $73$ difficult QSDP-QAP problems to an accuracy of $10^{-6}$ efficiently, while the purely
first-order algorithm {\sc Qsdpnal}-Phase I can only solve $2$ of the problems (chr20a and tai25a) to required accuracy.
The superior numerical performance of \QSDPNAL over {\sc Qsdpnal}-Phase I clearly demonstrates
 the importance and necessity of our proposed two-phase algorithm
for which second-order information is incorporated in the inexact augmented
Lagrangian algorithm in Phase II.

\begin{footnotesize}
\begin{longtable}{| c c | c |  c | c| c|c|}
\caption{The performance of \QSDPNAL and {\sc Qsdpnal}-Phase I on QSDP-QAP problems
(accuracy $= 10^{-6}$). In the table, ``a'' stands for \QSDPNAL and ``b'' stands for {\sc Qsdpnal}-Phase I, respectively.  The computation time is in the format of ``hours:minutes:seconds''.}\label{table:qsdpnal}
\\
\hline
 \mc{2}{|c|}{} &\mc{1}{c|}{} &\mc{1}{c|}{}&\mc{1}{c|}{}&\mc{1}{c|}{}&\mc{1}{c|}{}\\[-5pt]
\mc{2}{|c|}{} & \mc{1}{c|}{iter.a} &\mc{1}{c|}{iter.b} &\mc{1}{c|}{$\eta_{\textup{qsdp}}$}
&\mc{1}{c|}{$\eta_\textup{gap}$}&\mc{1}{c|}{time}\\[2pt] \hline
\mc{1}{|@{}c@{}}{problem} &\mc{1}{@{}c@{}|}{$m_E$ $|$ $n$} &\mc{1}{c|}{it (subs)$|$itSCB}
&\mc{1}{c|}{}&\mc{1}{c|}{a$|$b}&\mc{1}{c|}{a$|$b}
&\mc{1}{c|}{a$|$b}\\ \hline
\endhead

chr12a 
&232  $;$  144 	 &45  (239) $|$ 1969    &50000 	 &   9.9-7 $|$    {\green{ 2.2-6}}	 &   {\blue{ -6.0-6}} $|$    {\red{ -2.5-5}} 	 &41 $|$ 6:34\\[2pt] 
\hline 
chr12b 
&232  $;$  144 	 &56  (324) $|$ 2428    &50000 	 &   9.9-7 $|$    {\green{ 3.7-6}}	 &   {\red{ -2.0-5}} $|$    {\red{ -6.0-5}} 	 &50 $|$ 6:27\\[2pt] 
\hline 
chr12c 
&232  $;$  144 	 &56  (358) $|$ 2201    &50000 	 &   9.9-7 $|$    {\green{ 4.5-6}}	 &   {\red{ -1.6-5}} $|$    {\red{ -6.1-5}} 	 &46 $|$ 6:27\\[2pt] 
\hline 
chr15a 
&358  $;$  225 	 &84  (648) $|$ 2866    &50000 	 &   9.9-7 $|$    {\blue{ 5.7-6}}	 &   {\red{ -1.6-5}} $|$    {\bf -1.0-4} 	 &2:25 $|$ 12:23\\[2pt] 
\hline 
chr15b 
&358  $;$  225 	 &90  (584) $|$ 4700    &50000 	 &   9.9-7 $|$    {\blue{ 7.3-6}}	 &   {\red{ -1.3-5}} $|$    {\bf -1.3-4} 	 &3:07 $|$ 12:31\\[2pt] 
\hline 
chr15c 
&358  $;$  225 	 &65  (425) $|$ 2990    &50000 	 &   9.9-7 $|$    {\blue{ 6.8-6}}	 &   {\red{ -2.4-5}} $|$    {\red{ -9.7-5}} 	 &1:58 $|$ 12:40\\[2pt] 
\hline 
chr18a 
&511  $;$  324 	 &256  (1957) $|$ 6003    &50000 	 &   7.3-7 $|$    {\blue{ 6.1-6}}	 &   {\red{ -1.8-5}} $|$    {\bf -1.3-4} 	 &14:34 $|$ 22:26\\[2pt] 
\hline 
chr18b 
&511  $;$  324 	 &86  (565) $|$ 3907    &50000 	 &   9.9-7 $|$    {\blue{ 8.4-6}}	 &   {\red{ -1.8-5}} $|$    {\bf -1.6-4} 	 &5:26 $|$ 22:19\\[2pt] 
\hline 
chr20a 
&628  $;$  400 	 &39  (274) $|$ 1751    &4133 	 &   9.5-7 $|$    9.7-7	 &   {\red{ -3.3-5}} $|$    {\red{ -3.4-5}} 	 &4:50 $|$ 5:46\\[2pt] 
\hline 
chr20b 
&628  $;$  400 	 &72  (490) $|$ 4044    &50000 	 &   9.6-7 $|$    {\blue{ 9.1-6}}	 &   {\red{ -3.7-5}} $|$    {\bf -1.4-4} 	 &12:30 $|$ 58:57\\[2pt] 
\hline 
chr20c 
&628  $;$  400 	 &144  (981) $|$ 5242    &50000 	 &   9.9-7 $|$    {\red{ 1.4-5}}	 &   {\red{ -3.1-5}} $|$    {\bf -3.1-4} 	 &21:56 $|$ 55:41\\[2pt] 
\hline 
chr22a 
&757  $;$  484 	 &67  (473) $|$ 2804    &50000 	 &   9.9-7 $|$    {\green{ 5.0-6}}	 &   {\red{ -1.0-5}} $|$    {\red{ -7.6-5}} 	 &13:49 $|$   1:21:01\\[2pt] 
\hline 
chr22b 
&757  $;$  484 	 &69  (505) $|$ 3581    &50000 	 &   9.9-7 $|$    {\blue{ 6.5-6}}	 &   {\red{ -1.2-5}} $|$    {\bf -1.1-4} 	 &17:12 $|$   1:19:48\\[2pt] 
\hline 
els19 
&568  $;$  361 	 &43  (403) $|$ 2437    &50000 	 &   9.8-7 $|$    {\green{ 1.2-6}}	 &   {\green{ -4.2-6}} $|$    {\blue{ -8.6-6}} 	 &4:39 $|$   1:04:08\\[2pt] 
\hline 
esc16a 
&406  $;$  256 	 &86  (506) $|$ 5446    &50000 	 &   9.9-7 $|$    {\blue{ 8.3-6}}	 &   {\red{ -2.5-5}} $|$    {\bf -1.3-4} 	 &4:47 $|$ 16:54\\[2pt] 
\hline 
esc16b 
&406  $;$  256 	 &157  (1425) $|$ 9222    &50000 	 &   9.9-7 $|$    {\red{ 1.3-5}}	 &   {\red{ -3.7-5}} $|$    {\bf -2.7-4} 	 &11:18 $|$ 16:56\\[2pt] 
\hline 
esc16c 
&406  $;$  256 	 &188  (1404) $|$ 13806    &50000 	 &   9.9-7 $|$    {\red{ 1.2-5}}	 &   {\red{ -4.6-5}} $|$    {\bf -3.5-4} 	 &14:29 $|$ 16:57\\[2pt] 
\hline 
esc16d 
&406  $;$  256 	 &101  (603) $|$ 8043    &50000 	 &   9.9-7 $|$    {\green{ 4.8-6}}	 &   {\red{ -1.1-5}} $|$    {\red{ -7.4-5}} 	 &6:13 $|$ 16:55\\[2pt] 
\hline 
esc16e 
&406  $;$  256 	 &110  (847) $|$ 4286    &50000 	 &   9.9-7 $|$    {\green{ 4.8-6}}	 &   {\red{ -1.4-5}} $|$    {\red{ -5.6-5}} 	 &5:50 $|$ 16:35\\[2pt] 
\hline 
esc16g 
&406  $;$  256 	 &85  (581) $|$ 3818    &50000 	 &   9.9-7 $|$    {\blue{ 7.2-6}}	 &   {\red{ -2.2-5}} $|$    {\red{ -9.4-5}} 	 &4:23 $|$ 16:44\\[2pt] 
\hline 
esc16h 
&406  $;$  256 	 &228  (1732) $|$ 11733    &50000 	 &   8.6-7 $|$    {\blue{ 8.6-6}}	 &   {\blue{ -9.3-6}} $|$    {\red{ -8.7-5}} 	 &13:58 $|$ 16:21\\[2pt] 
\hline 
esc16i 
&406  $;$  256 	 &41  (307) $|$ 3165    &50000 	 &   9.6-7 $|$    {\green{ 4.6-6}}	 &   {\red{ -2.0-5}} $|$    {\red{ -6.0-5}} 	 &2:28 $|$ 16:54\\[2pt] 
\hline 
esc16j 
&406  $;$  256 	 &163  (1179) $|$ 5603    &50000 	 &   9.9-7 $|$    {\blue{ 6.6-6}}	 &   {\red{ -2.0-5}} $|$    {\bf -1.0-4} 	 &8:03 $|$ 16:24\\[2pt] 
\hline 
esc32b 
&1582  $;$  1024 	 &80  (456) $|$ 5026    &50000 	 &   9.9-7 $|$    {\bf 1.5-4}	 &   {\red{ -2.9-5}} $|$    {\bf -5.3-4} 	 &  1:53:02 $|$   7:42:25\\[2pt] 
\hline 
esc32c 
&1582  $;$  1024 	 &105  (667) $|$ 4203    &50000 	 &   8.9-7 $|$    {\blue{ 7.2-6}}	 &   {\red{ -1.0-5}} $|$    {\red{ -7.3-5}} 	 &  2:40:09 $|$   7:36:49\\[2pt] 
\hline 
esc32d 
&1582  $;$  1024 	 &141  (909) $|$ 4852    &50000 	 &   9.9-7 $|$    {\blue{ 5.7-6}}	 &   {\blue{ -8.6-6}} $|$    {\red{ -6.4-5}} 	 &  3:09:45 $|$   7:19:17\\[2pt] 
\hline 
had12 
&232  $;$  144 	 &53  (320) $|$ 2903    &50000 	 &   9.3-7 $|$    {\green{ 3.3-6}}	 &   {\blue{ -7.2-6}} $|$    {\red{ -2.7-5}} 	 &55 $|$ 6:52\\[2pt] 
\hline 
had14 
&313  $;$  196 	 &60  (443) $|$ 3634    &50000 	 &   9.9-7 $|$    {\green{ 4.5-6}}	 &   {\blue{ -6.3-6}} $|$    {\red{ -2.8-5}} 	 &1:50 $|$ 11:21\\[2pt] 
\hline 
had16 
&406  $;$  256 	 &208  (1616) $|$ 7604    &50000 	 &   8.1-7 $|$    {\blue{ 9.9-6}}	 &   {\green{ -4.7-6}} $|$    {\red{ -9.1-5}} 	 &10:49 $|$ 16:54\\[2pt] 
\hline 
had18 
&511  $;$  324 	 &82  (537) $|$ 4367    &50000 	 &   9.9-7 $|$    {\blue{ 9.2-6}}	 &   {\red{ -1.5-5}} $|$    {\red{ -7.3-5}} 	 &6:10 $|$ 24:46\\[2pt] 
\hline 
had20 
&628  $;$  400 	 &121  (848) $|$ 5024    &50000 	 &   9.5-7 $|$    {\red{ 1.0-5}}	 &   {\red{ -1.2-5}} $|$    {\bf -1.0-4} 	 &20:30 $|$ 51:08\\[2pt] 
\hline 
kra30a 
&1393  $;$  900 	 &107  (674) $|$ 4665    &50000 	 &   9.5-7 $|$    {\blue{ 6.5-6}}	 &   {\red{ -6.7-5}} $|$    {\bf -1.7-4} 	 &  1:51:20 $|$   7:23:56\\[2pt] 
\hline 
kra30b 
&1393  $;$  900 	 &107  (674) $|$ 4853    &50000 	 &   9.9-7 $|$    {\blue{ 6.5-6}}	 &   {\red{ -5.7-5}} $|$    {\bf -1.7-4} 	 &  2:00:03 $|$   8:02:08\\[2pt] 
\hline 
kra32 
&1582  $;$  1024 	 &106  (636) $|$ 6875    &50000 	 &   9.9-7 $|$    {\blue{ 7.4-6}}	 &   {\red{ -3.6-5}} $|$    {\bf -1.6-4} 	 &  2:47:33 $|$   10:28:26\\[2pt] 
\hline 
lipa30a 
&1393  $;$  900 	 &64  (451) $|$ 2924    &50000 	 &   9.9-7 $|$    {\blue{ 5.6-6}}	 &   {\blue{ -6.8-6}} $|$    {\red{ -3.1-5}} 	 &  1:10:12 $|$   6:52:34\\[2pt] 
\hline 
lipa30b 
&1393  $;$  900 	 &257  (1918) $|$ 7507    &50000 	 &   9.9-7 $|$    {\blue{ 7.0-6}}	 &   {\green{ -1.9-6}} $|$    {\bf -1.9-4} 	 &  4:28:38 $|$   7:18:10\\[2pt] 
\hline 
lipa40a 
&2458  $;$  1600 	 &51  (349) $|$ 2193    &50000 	 &   7.7-7 $|$    {\green{ 4.2-6}}	 &   {\green{ -2.3-6}} $|$    {\red{ -2.0-5}} 	 &  3:03:58 $|$   23:53:41\\[2pt] 
\hline 
lipa40b 
&2458  $;$  1600 	 &156  (1339) $|$ 4750    &50000 	 &   9.1-7 $|$    {\green{ 3.9-6}}	 &   {\blue{ 6.0-6}} $|$    {\red{ -8.9-5}} 	 &  9:36:10 $|$   18:57:01\\[2pt] 
\hline 
nug12 
&232  $;$  144 	 &84  (478) $|$ 4068    &50000 	 &   9.8-7 $|$    {\blue{ 5.4-6}}	 &   {\red{ -3.0-5}} $|$    {\bf -1.1-4} 	 &1:32 $|$ 6:34\\[2pt] 
\hline 
nug14 
&313  $;$  196 	 &93  (610) $|$ 4953    &50000 	 &   9.7-7 $|$    {\blue{ 6.9-6}}	 &   {\red{ -2.6-5}} $|$    {\bf -1.1-4} 	 &3:12 $|$ 10:27\\[2pt] 
\hline 
nug15 
&358  $;$  225 	 &102  (660) $|$ 5627    &50000 	 &   7.0-7 $|$    {\red{ 1.1-5}}	 &   {\red{ -2.2-5}} $|$    {\bf -1.7-4} 	 &4:12 $|$ 12:42\\[2pt] 
\hline 
nug16a 
&406  $;$  256 	 &86  (530) $|$ 4945    &50000 	 &   9.9-7 $|$    {\blue{ 7.2-6}}	 &   {\red{ -2.3-5}} $|$    {\bf -1.1-4} 	 &4:39 $|$ 18:17\\[2pt] 
\hline 
nug16b 
&406  $;$  256 	 &97  (631) $|$ 4777    &50000 	 &   9.9-7 $|$    {\red{ 1.2-5}}	 &   {\red{ -2.5-5}} $|$    {\bf -2.0-4} 	 &5:19 $|$ 19:06\\[2pt] 
\hline 
nug17 
&457  $;$  289 	 &110  (772) $|$ 5365    &50000 	 &   9.9-7 $|$    {\red{ 1.3-5}}	 &   {\red{ -2.4-5}} $|$    {\bf -1.8-4} 	 &7:41 $|$ 22:47\\[2pt] 
\hline 
nug18 
&511  $;$  324 	 &85  (559) $|$ 4367    &50000 	 &   9.9-7 $|$    {\blue{ 6.1-6}}	 &   {\red{ -3.3-5}} $|$    {\red{ -9.9-5}} 	 &6:10 $|$ 26:20\\[2pt] 
\hline 
nug20 
&628  $;$  400 	 &114  (746) $|$ 5220    &50000 	 &   9.9-7 $|$    {\blue{ 8.6-6}}	 &   {\red{ -2.3-5}} $|$    {\bf -1.3-4} 	 &19:25 $|$ 55:36\\[2pt] 
\hline 
nug21 
&691  $;$  441 	 &84  (569) $|$ 4322    &50000 	 &   9.7-7 $|$    {\blue{ 6.8-6}}	 &   {\red{ -4.0-5}} $|$    {\bf -1.1-4} 	 &18:48 $|$   1:09:39\\[2pt] 
\hline 
nug22 
&757  $;$  484 	 &121  (822) $|$ 5822    &50000 	 &   9.6-7 $|$    {\blue{ 8.3-6}}	 &   {\red{ -4.1-5}} $|$    {\bf -1.3-4} 	 &34:03 $|$   2:08:02\\[2pt] 
\hline 
nug24 
&898  $;$  576 	 &89  (542) $|$ 4345    &50000 	 &   9.9-7 $|$    {\blue{ 6.5-6}}	 &   {\red{ -3.2-5}} $|$    {\bf -1.1-4} 	 &34:08 $|$   3:04:07\\[2pt] 
\hline 
nug25 
&973  $;$  625 	 &129  (860) $|$ 5801    &50000 	 &   9.9-7 $|$    {\blue{ 6.9-6}}	 &   {\red{ -2.2-5}} $|$    {\bf -1.1-4} 	 &  1:09:12 $|$   3:41:22\\[2pt] 
\hline 
nug27 
&1132  $;$  729 	 &148  (951) $|$ 8576    &50000 	 &   9.9-7 $|$    {\blue{ 8.3-6}}	 &   {\red{ -2.6-5}} $|$    {\bf -1.3-4} 	 &  2:09:22 $|$   4:59:06\\[2pt] 
\hline 
nug28 
&1216  $;$  784 	 &119  (758) $|$ 6389    &50000 	 &   9.7-7 $|$    {\blue{ 7.8-6}}	 &   {\red{ -2.9-5}} $|$    {\bf -1.1-4} 	 &  1:43:52 $|$   5:50:50\\[2pt] 
\hline 
nug30 
&1393  $;$  900 	 &105  (777) $|$ 4912    &50000 	 &   9.9-7 $|$    {\blue{ 9.3-6}}	 &   {\red{ -2.6-5}} $|$    {\bf -1.2-4} 	 &  2:08:56 $|$   7:40:43\\[2pt] 
\hline 
rou12 
&232  $;$  144 	 &78  (418) $|$ 6600    &50000 	 &   9.9-7 $|$    {\green{ 4.4-6}}	 &   {\red{ -2.2-5}} $|$    {\red{ -9.3-5}} 	 &2:13 $|$ 6:36\\[2pt] 
\hline 
rou15 
&358  $;$  225 	 &106  (639) $|$ 5952    &50000 	 &   9.9-7 $|$    {\blue{ 5.4-6}}	 &   {\red{ -2.7-5}} $|$    {\bf -1.0-4} 	 &4:12 $|$ 13:02\\[2pt] 
\hline 
rou20 
&628  $;$  400 	 &65  (359) $|$ 4238    &50000 	 &   9.9-7 $|$    {\green{ 3.6-6}}	 &   {\red{ -2.5-5}} $|$    {\red{ -6.1-5}} 	 &11:53 $|$   1:00:07\\[2pt] 
\hline 
scr12 
&232  $;$  144 	 &56  (295) $|$ 2205    &50000 	 &   9.9-7 $|$    {\green{ 3.2-6}}	 &   {\blue{ -7.5-6}} $|$    {\red{ -4.0-5}} 	 &43 $|$ 6:57\\[2pt] 
\hline 
scr15 
&358  $;$  225 	 &121  (769) $|$ 5730    &50000 	 &   7.4-7 $|$    {\red{ 1.0-5}}	 &   {\red{ -2.1-5}} $|$    {\bf -1.9-4} 	 &4:39 $|$ 12:45\\[2pt] 
\hline 
scr20 
&628  $;$  400 	 &89  (590) $|$ 5621    &50000 	 &   9.9-7 $|$    {\blue{ 8.3-6}}	 &   {\red{ -4.1-5}} $|$    {\bf -1.6-4} 	 &18:49 $|$   1:01:19\\[2pt] 
\hline 
tai12a 
&232  $;$  144 	 &110  (807) $|$ 6090    &50000 	 &   9.9-7 $|$    {\blue{ 8.9-6}}	 &   {\red{ -1.8-5}} $|$    {\bf -1.3-4} 	 &2:40 $|$ 6:15\\[2pt] 
\hline 
tai12b 
&232  $;$  144 	 &123  (856) $|$ 6323    &50000 	 &   8.6-7 $|$    {\blue{ 7.3-6}}	 &   {\red{ -2.5-5}} $|$    {\bf -1.1-4} 	 &2:43 $|$ 6:19\\[2pt] 
\hline 
tai15a 
&358  $;$  225 	 &67  (405) $|$ 4301    &50000 	 &   9.4-7 $|$    {\green{ 3.0-6}}	 &   {\red{ -2.8-5}} $|$    {\red{ -5.9-5}} 	 &2:48 $|$ 13:09\\[2pt] 
\hline 
tai17a 
&457  $;$  289 	 &95  (569) $|$ 6142    &50000 	 &   9.9-7 $|$    {\green{ 3.8-6}}	 &   {\red{ -2.0-5}} $|$    {\red{ -6.5-5}} 	 &6:33 $|$ 19:23\\[2pt] 
\hline 
tai20a 
&628  $;$  400 	 &87  (498) $|$ 4762    &50000 	 &   9.7-7 $|$    {\green{ 3.0-6}}	 &   {\red{ -2.2-5}} $|$    {\red{ -5.4-5}} 	 &15:00 $|$   1:00:55\\[2pt] 
\hline 
tai25a 
&973  $;$  625 	 &25  (138) $|$ 3438    &10084 	 &   9.9-7 $|$    9.9-7	 &   {\blue{ 9.4-6}} $|$    {\red{ -1.5-5}} 	 &17:46 $|$ 47:55\\[2pt] 
\hline 
tai25b 
&973  $;$  625 	 &164  (1219) $|$ 7803    &50000 	 &   9.8-7 $|$    {\red{ 1.4-5}}	 &   {\red{ -4.8-5}} $|$    {\bf -2.6-4} 	 &  1:33:02 $|$   3:33:13\\[2pt] 
\hline 
tai30a 
&1393  $;$  900 	 &97  (546) $|$ 4270    &50000 	 &   7.8-7 $|$    {\green{ 2.8-6}}	 &   {\red{ -1.4-5}} $|$    {\red{ -4.1-5}} 	 &  1:22:01 $|$   7:21:02\\[2pt] 
\hline 
tai30b 
&1393  $;$  900 	 &162  (1229) $|$ 6845    &50000 	 &   9.9-7 $|$    {\red{ 1.2-5}}	 &   {\red{ -3.3-5}} $|$    {\bf -2.1-4} 	 &  3:10:41 $|$   7:21:29\\[2pt] 
\hline 
tai35a 
&1888  $;$  1225 	 &95  (537) $|$ 5781    &50000 	 &   9.9-7 $|$    {\green{ 2.8-6}}	 &   {\blue{ -9.2-6}} $|$    {\red{ -3.5-5}} 	 &  3:34:56 $|$   15:04:40\\[2pt] 
\hline 
tai35b 
&1888  $;$  1225 	 &152  (1195) $|$ 5303    &50000 	 &   9.6-7 $|$    {\red{ 1.2-5}}	 &   {\red{ -3.8-5}} $|$    {\bf -1.9-4} 	 &  5:56:59 $|$   13:37:28\\[2pt] 
\hline 
tai40a 
&2458  $;$  1600 	 &79  (398) $|$ 6381    &50000 	 &   9.1-7 $|$    {\green{ 3.0-6}}	 &   {\red{ -1.7-5}} $|$    {\red{ -3.4-5}} 	 &  6:01:56 $|$   23:07:21\\[2pt] 
\hline 
tho30 
&1393  $;$  900 	 &116  (761) $|$ 5468    &50000 	 &   8.9-7 $|$    {\blue{ 8.0-6}}	 &   {\red{ -3.1-5}} $|$    {\bf -1.3-4} 	 &  2:15:55 $|$   7:39:54\\[2pt] 
\hline 
tho40 
&2458  $;$  1600 	 &122  (762) $|$ 3834    &50000 	 &   9.9-7 $|$    {\blue{ 7.4-6}}	 &   {\red{ -2.9-5}} $|$    {\bf -1.0-4} 	 &  6:28:52 $|$   26:40:25\\[2pt] 
\hline

\end{longtable}
\end{footnotesize}

\subsection{Evaluation of \QSDPNAL on instances generated from sensor network localization problems}

We also test the QSDP problems arising from the following sensor network localization problems with $m$ anchors and $l$ sensors:
\begin{equation}\label{snl}
\begin{array}{rl}
   \min_{u_1,\ldots,u_l\in\Re^d} \Big\{ \frac{1}{2}\sum_{(i,j)\in\cN} \big(\norm{u_i - u_j}^2 - d_{ij}^2\big)^2 \,\mid\,
   \norm{u_i - a_k}^2 = d_{ik}^2,\, (i,k)\in\cM\Big\},
  \end{array}
\end{equation}
where the location of each anchor $a_k\in \Re^d$, $k=1,\ldots,m$ is known, and the location of each sensor $u_i\in\Re^d$, $i=1,\ldots,l$, is to be determined.
The distance measures $\{ d_{ij} \mid (i,j)\in\cN\}$ and $\{ d_{ik}\mid (i,k)\in\cM\}$  are known pair-wise distances between sensor-sensor pairs and sensor-anchor pairs, respectively. Note that our model \eqref{snl} is a variant of the model studied in \cite{biswas2006snl}. Let $U = [u_1\ u_2\ \ldots u_l]\in\Re^{d\times l}$ be the position matrix that needs to be determined. We know that
\[\norm{u_i - u_j}^2 = e_{ij}^T U^T U e_{ij}, \qquad
\norm{u_i-a_k}^2 = a_{ik}^T [U \ I_d]^T[U \ I_d]a_{ik},\]
where $e_{ij} = e_i - e_j$ and $a_{ik} = [e_i; -a_k]$. Here,
$e_i$ is the $i$th unit vector in $\Re^l$,  and $I_d$ is the $d\times d$ identity matrix.
Let $g_{ik} = a_{ik}$ for $(i,k)\in\cM$, $g_{ij} = [e_{ij}; {\bf 0}_m]$ for $(i,j)\in\cN$,
and
\[V = U^T U,\qquad X = [ V \ U^T; U \ I_d ]\in \cS^{(d+l)\times(d+l)}. \] Following the same approach in \cite{biswas2006snl}, we can obtain the following QSDP relaxation model with regularization term for \eqref{snl}
\begin{equation}
  \label{snl_QSDP}
  \begin{array}{rl}
   \min & \frac{1}{2}\sum_{(i,j)\in\cN} \big(g_{ij}^T X g_{ij} - d_{ij}^2\big)^2 - \lambda\inprod{I_{n+d} - aa^T}{X}
\\[8pt]
  {\rm s.t.} &  g_{ik}^T X g_{ik} = d_{ik}^2,\, (i,k)\in\cM, \quad X\succeq 0,
  \end{array}
\end{equation}
  where $\lambda$ is a given positive regularzation parameter and $a = [\hat e; \hat a]$ with $\hat e = e/\sqrt{l+m}$ and $\hat a = \sum_{k=1}^m a_k/\sqrt{l+m}$. Here $e\in\Re^n$ is the vector of all ones.

The test examples are generated in the following {manner}. We first randomly generate $l$ points $\{\hat x_{i}\in\Re^d \mid i=1,\ldots,l\}$ in $[-0.5,0.5]^d$. Then, the edge set $\cN$ is generated by considering only pairs of points that have distances less than a given positive number $R$, i.e.,
\[\cN = \{(i,j)\,\mid\, \norm{\hat u_i - \hat u_j}\le R, \ 1\le i < j\le l\}.\]
Given $m$ anchors $\{ a_{k}\in\Re^d \mid k=1,\ldots,m\}$, the edge set $\cM$ is similarly given by
\[\cM = \{(i,k)\,\mid\, \norm{\hat u_i - a_k}\le R, \ 1\le i \le l, \ 1\le k\le m\}.\]
We also assume that the {observed} distances $d_{ij}$ are perturbed by random noises $\varepsilon_{ij}$ as follows:
\[d_{ij} = \hat d_{ij}|1+\tau\varepsilon_{ij}|,\quad (i,j)\in\cN,\]
where $\hat d_{ij}$ is the true distance between point $i$ and $j$, $\varepsilon_{ij}$ are assumed to be independent standard Normal random variables, $\tau$ is the noise parameter.
For the numerical experiments,
we generate $10$ instances where the number of sensors $l$ ranges from 250 to 1500 and the dimension $d$ is set to be 2 or 3. W set the noise factor $\tau = 10\%$. The 4 anchors for the two dimensional case ($d=2$) are placed at
\[(\pm 0.3,\pm 0.3),\]
and the positions of the anchors for the three dimensional case ($d=3$) are given by
\[\left(
    \begin{array}{cccc}
      1/3 & 2/3 & 2/3 & 1/3 \\
      1/3 & 2/3 & 1/3 & 2/3 \\
      1/3 & 1/3 & 2/3 & 2/3 \\
    \end{array}
  \right) - 0.5.
\]

\begin{footnotesize}
\begin{longtable}{| c c | c |  c | c| c|c|}
\caption{The performance of \QSDPNAL and {\sc Qsdpnal}-Phase I on the sensor network localization problems (dual of \eqref{snl_QSDP})
(accuracy $= 10^{-6}$). In the table, ``a'' stands for \QSDPNAL and ``b'' stands for {\sc Qsdpnal}-Phase I, respectively.  The computation time is in the format of ``hours:minutes:seconds''.}\label{table:snl}
\\
\hline
 \mc{2}{|c|}{} &\mc{1}{c|}{} &\mc{1}{c|}{}&\mc{1}{c|}{}&\mc{1}{c|}{}&\mc{1}{c|}{}\\[-5pt]
\mc{2}{|c|}{} & \mc{1}{c|}{iter.a} &\mc{1}{c|}{iter.b} &\mc{1}{c|}{$\eta_{\textup{qsdp}}$}
&\mc{1}{c|}{$\eta_\textup{gap}$}&\mc{1}{c|}{time}\\[2pt] \hline
\mc{1}{|@{}c@{}}{$d$} &\mc{1}{@{}c@{}|}{$m_E$ $|$ $n$ $|$ $R$} &\mc{1}{c|}{it (subs)$|$itSCB}
&\mc{1}{c|}{}&\mc{1}{c|}{a$|$b}&\mc{1}{c|}{a$|$b}
&\mc{1}{c|}{a$|$b}\\ \hline
\endhead

2 
&452  $|$  252  $|$  0.50 	 &12 (74) $|$ 652    &24049 	 &   7.1-7 $|$    9.9-7	 &   -9.0-7 $|$    6.8-7 	 &47 $|$ 6:10\\[2pt] 
\hline 
2 
&548  $|$  502  $|$  0.36 	 &12 (62) $|$ 1000    &12057 	 &   7.6-7 $|$    9.9-7	 &   {\blue{ -9.2-6}} $|$    {\blue{ -9.1-6}} 	 &1:49 $|$ 17:25\\[2pt] 
\hline 
2 
&633  $|$  802  $|$  0.28 	 &17 (85) $|$ 1000    &27361 	 &   3.0-7 $|$    9.9-7	 &   {\green{ -2.4-6}} $|$    {\blue{ -9.9-6}} 	 &5:42 $|$   1:59:16\\[2pt] 
\hline 
2 
&684  $|$  1002  $|$  0.25 	 &17 (94) $|$ 1000    &50000 	 &   4.1-7 $|$    {\red{ 1.4-5}}	 &   {\green{ -2.6-6}} $|$    -3.0-7 	 &10:18 $|$   6:16:37\\[2pt] 
\hline 
2 
&781  $|$  1502  $|$  0.21 	 &21 (104) $|$ 1000    &50000 	 &   3.6-7 $|$    {\bf 9.5-4}	 &   {\blue{ -6.3-6}} $|$    {\bf 5.1-3} 	 &23:05 $|$   13:47:39\\[2pt] 
\hline 
2 
&774  $|$  2002  $|$  0.18 	 &29 (156) $|$ 1000    &50000 	 &   7.3-7 $|$    {\bf 2.1-3}	 &   {\green{ -3.8-6}} $|$    {\bf 1.4-2} 	 &49:53 $|$   23:20:28\\[2pt] 
\hline 
3 
&395  $|$  253  $|$  0.49 	 &11 (31) $|$ 408    &1487 	 &   9.8-7 $|$    9.7-7	 &   {\green{ -1.4-6}} $|$    1.3-7 	 &06 $|$ 18\\[2pt] 
\hline 
3 
&503  $|$  503  $|$  0.39 	 &14 (61) $|$ 877    &7882 	 &   3.5-7 $|$    9.9-7	 &   {\green{ -1.1-6}} $|$    {\green{ 2.5-6}} 	 &1:46 $|$ 7:18\\[2pt] 
\hline 
3 
&512  $|$  803  $|$  0.33 	 &15 (85) $|$ 1000    &10579 	 &   7.7-7 $|$    9.9-7	 &   {\green{ -1.3-6}} $|$    4.2-7 	 &7:13 $|$ 26:15\\[2pt] 
\hline 
3 
&513  $|$  1003  $|$  0.31 	 &16 (71) $|$ 1000    &14025 	 &   2.5-7 $|$    9.9-7	 &   -7.6-7 $|$    {\green{ 4.2-6}} 	 &8:46 $|$   1:02:07\\[2pt] 
\hline 
3 
&509  $|$  1503  $|$  0.27 	 &19 (83) $|$ 1000    &23328 	 &   8.6-7 $|$    9.9-7	 &   {\green{ -4.5-6}} $|$    {\blue{ 7.2-6}} 	 &28:34 $|$   4:13:07\\[2pt] 
\hline 
3 
&505  $|$  2003  $|$  0.24 	 &19 (97) $|$ 1000    &50000 	 &   9.5-7 $|$    {\bf 3.3-4}	 &   {\red{ -1.4-5}} $|$    {\bf -3.7-4} 	 &49:28 $|$   16:45:24\\[2pt] 
\hline

\end{longtable}
\end{footnotesize}

{Let $\cN_i = \{ p \mid (i,p)\in \cN\}$ be the set of neighbors of the $i$th sensor.}
To further test our algorithm {\sc Qsdpnal}, we also generate the following valid inequalities and add them to problem \eqref{snl}
\[\norm{\hat u_i - \hat u_j} \ge R,\, \forall\, (i,j)\in\widehat\cN,\]
where {
$ \widehat{\cN} = \bigcup_{i=1}^n \{ (i,j) \mid j \in \cN_p\backslash \cN_i \; \mbox{for some} \;
p\in \cN_i\}.$
}
Then, we obtain the following QSDP relaxation model
\begin{equation}
  \label{snlI_QSDP}
  \begin{array}{rl}
    \min & \frac{1}{2}\sum_{(i,j)\in\cN} \big(g_{ij}^T X g_{ij} - d_{ij}^2\big)^2 - \lambda\inprod{I_{n+d} - aa^T}{X}\\[8pt]
  {\rm s.t.} &  g_{ik}^T X g_{ik} = d_{ik}^2,\, (i,k)\in\cM, \\[5pt]

  &g_{ij}^T X g_{ij} \ge R^2,\, (i,j)\in\widehat\cN,\quad X\succeq 0.
  \end{array}
\end{equation}

In Table \ref{table:snl} and \ref{table:snlI},
we present the detail
numerical results for \QSDPNAL and {\sc Qsdpnal}-Phase I in solving some instances of problem \eqref{snl_QSDP} and \eqref{snlI_QSDP}, respectively. Clearly, \QSDPNAL outperforms the purely first-order algorithm {\sc Qsdpnal}-Phase I by a significant margin.
This superior numerical performance of \QSDPNAL over {\sc Qsdpnal}-Phase I again demonstrates the importance and necessity of our proposed two-phase framework.

\begin{footnotesize}
\begin{longtable}{| c c | c |  c | c| c|c|}
\caption{The performance of \QSDPNAL and {\sc Qsdpnal}-Phase I on the sensor network localization problems (dual of \eqref{snlI_QSDP})
(accuracy $= 10^{-6}$). In the table, ``a'' stands for \QSDPNAL and ``b'' stands for {\sc Qsdpnal}-Phase I, respectively.  The computation time is in the format of ``hours:minutes:seconds''.}\label{table:snlI}
\\
\hline
 \mc{2}{|c|}{} &\mc{1}{c|}{} &\mc{1}{c|}{}&\mc{1}{c|}{}&\mc{1}{c|}{}&\mc{1}{c|}{}\\[-5pt]
\mc{2}{|c|}{} & \mc{1}{c|}{iter.a} &\mc{1}{c|}{iter.b} &\mc{1}{c|}{$\eta_{\textup{qsdp}}$}
&\mc{1}{c|}{$\eta_\textup{gap}$}&\mc{1}{c|}{time}\\[2pt] \hline
\mc{1}{|@{}c@{}}{$d$} &\mc{1}{@{}c@{}|}{$m_E;m_I$ $|$ $n$ $|$ $R$} &\mc{1}{c|}{it (subs)$|$itSCB}
&\mc{1}{c|}{}&\mc{1}{c|}{a$|$b}&\mc{1}{c|}{a$|$b}
&\mc{1}{c|}{a$|$b}\\ \hline
\endhead

2 
&452  $;$ 14402  $|$  252  $|$  0.50 	 &15  (119) $|$ 603    &50000 	 &   6.4-7 $|$    {\green{ 1.9-6}}	 &   -9.1-7 $|$    -5.1-8 	 &1:50 $|$ 24:28\\[2pt] 
\hline 
2 
&548  $;$ 55849  $|$  502  $|$  0.36 	 &16  (180) $|$ 1357    &29565 	 &   4.4-7 $|$    9.9-7	 &   -2.2-7 $|$    {\green{ -2.3-6}} 	 &11:28 $|$   1:09:24\\[2pt] 
\hline 
2 
&633  $;$ 118131  $|$  802  $|$  0.28 	 &15  (226) $|$ 2330    &36651 	 &   6.3-7 $|$    9.9-7	 &   {\green{ -2.4-6}} $|$    {\blue{ -5.4-6}} 	 &  1:06:04 $|$   3:43:51\\[2pt] 
\hline 
2 
&684  $;$ 160157  $|$  1002  $|$  0.25 	 &20  (265) $|$ 3384    &50000 	 &   4.5-7 $|$    {\green{ 2.7-6}}	 &   {\green{ -1.7-6}} $|$    {\green{ 2.2-6}} 	 &  2:36:41 $|$   8:30:28\\[2pt] 
\hline 
2 
&724  $;$ 201375  $|$  1202  $|$  0.23 	 &21  (487) $|$ 3115    &50000 	 &   9.8-7 $|$    {\green{ 3.8-6}}	 &   {\blue{ 5.2-6}} $|$    {\green{ -2.8-6}} 	 &  6:36:28 $|$   11:57:56\\[2pt] 
\hline 
3 
&395  $;$ 16412  $|$  253  $|$  0.49 	 &12  (88) $|$ 471    &2897 	 &   3.1-7 $|$    9.8-7	 &   -3.3-7 $|$    -5.7-7 	 &46 $|$ 1:18\\[2pt] 
\hline 
3 
&503  $;$ 53512  $|$  503  $|$  0.39 	 &14  (136) $|$ 949    &11003 	 &   4.4-7 $|$    9.9-7	 &   -9.9-7 $|$    -3.1-7 	 &8:23 $|$ 20:26\\[2pt] 
\hline 
3 
&512  $;$ 104071  $|$  803  $|$  0.33 	 &17  (145) $|$ 1762    &14144 	 &   6.6-7 $|$    9.9-7	 &   {\green{ -2.6-6}} $|$    6.0-7 	 &32:17 $|$   1:09:10\\[2pt] 
\hline 
3 
&513  $;$ 139719  $|$  1003  $|$  0.31 	 &21  (198) $|$ 2406    &31832 	 &   5.1-7 $|$    9.9-7	 &   8.4-8 $|$    4.5-8 	 &  1:18:56 $|$   4:48:04\\[2pt] 
\hline 
3 
&526  $;$ 180236  $|$  1203  $|$  0.29 	 &21  (250) $|$ 2639    &19010 	 &   8.3-7 $|$    9.9-7	 &   -3.8-8 $|$    {\red{ 1.1-5}} 	 &  2:15:54 $|$   4:16:09\\[2pt] 
\hline

\end{longtable}
\end{footnotesize}

\section{Conclusions}\label{sec:conclu}

In this paper, we have designed a two-phase augmented Lagrangian {based method, called {\sc Qsdpnal},} for solving large scale convex quadratic semidefinite programming problems. The global  and local convergence rate analysis of our algorithm are based on the classic results of proximal point algorithms \cite{rockafellar1976monotone,rockafellar1976augmented}, together with the recent advances in second order variational analysis of  convex composite quadratic semidefinite programming \cite{cui2016on}.
By
devising ``smart'' numerical linear algebra, we have overcome various
challenging numerical difficulties encountered in the {efficient} implementation of {\sc Qsdpnal}. Numerical experiments on various large scale QSDPs have demonstrated  the efficiency and robustness of our proposed two-phase framework in obtaining accurate solutions. Specifically, for {well-posed} problems, our {\sc Qsdpnal}-Phase I is already powerful enough and it is not absolutely  necessary to execute {\sc Qsdpnal}-Phase II. On the other hand, for more difficult problems,
the purely first-order  {\sc Qsdpnal}-Phase I algorithm may stagnate because of {extremely}
slow local convergence. In contrast,  {with the activation of {\sc Qsdpnal}-Phase II which has second order information wisely incorporated, our {\sc Qsdpnal} algorithm can still obtain highly accurate solutions efficiently.}

{}

\end{document}